\documentclass[reqno]{amsart}

\usepackage{amsmath, amsfonts, amssymb, hyperref, tikz}
\usepackage{graphicx}
\usepackage{mathtools} 

\usetikzlibrary{matrix,arrows,shapes}
\usetikzlibrary{backgrounds}

{\theoremstyle{plain}
\newtheorem{theorem}{Theorem}[section]

\newtheorem{proposition}[theorem]{Proposition}
\newtheorem{lemma}[theorem]{Lemma}
\newtheorem{corollary}[theorem]{Corollary}
\newtheorem{question}[theorem]{Question}
}
{\theoremstyle{definition}
\newtheorem{definition}[theorem]{Definition}
\newtheorem*{definition*}{Definition}
\newtheorem{remark}[theorem]{Remark}

\newtheorem{example}[theorem]{Example}
}

\numberwithin{equation}{section}
\numberwithin{figure}{section}
 
\renewcommand{\theenumi}{\alph{enumi}}

\newcommand{\gen}[1]{K_{(#1)}}
\newcommand{\genx}[1]{K_{#1}}
\newcommand{\om}{\omega} 
\newcommand{\Om}{\Omega}
\newcommand{\kpw}{\gen{P,\omega}}
\newcommand{\kqt}{\gen{Q,\tau}}
\newcommand{\popf}{f}
\newcommand{\leqF}{\leq_F}
\newcommand{\leqM}{\leq_M}
\newcommand{\domleq}{\leq_{\mathrm{dom}}}
\newcommand{\pcovered}{\prec_P}

\newcommand{\des}{\mathrm{Des}}

\renewcommand{\L}{\mathcal{L}}
\newcommand{\Fsupp}{\mathrm{supp}_F}
\newcommand{\Msupp}{\mathrm{supp}_M}
\newcommand{\sd}[1]{P_{#1}}
\newcommand{\rotatep}[1]{(#1)^*}
\newcommand{\rotatex}[1]{#1^*}
\newcommand{\switch}[1]{\overline{(#1)}}
\newcommand{\switchx}[1]{\overline{#1}}
\newcommand{\jump}{\mathrm{jump}}
\newcommand{\lessthan}{S_<}
\newcommand{\ur}[2]{\mathcal{#1}[#2 \to #1_i]}
\newcommand{\urp}{\ur{P}{i}}
\newcommand{\urq}{\ur{Q}{i}}
\renewcommand{\P}{\mathcal{P}}
\newcommand{\CQ}{\mathcal{Q}}
\newcommand{\urpom}{(\urp, \Om)}
\newcommand{\kurp}{\genx{\urp}}

\newcommand{\I}{\mathcal{I}}
\newcommand{\card}[2]{#1_{\mathrm{#2}}}

\begin{document}
\title{Positivity among $P$-partition generating functions}

\author{Nathan R.\,T. Lesnevich}
\address{Department of Mathematics, Oklahoma State University, Stillwater, OK 74075, USA}
\email{\href{mailto:nlesnev@okstate.edu}{nlesnev@okstate.edu}}

\author{Peter R.\,W. McNamara}
\address{Department of Mathematics, Bucknell University, Lewisburg, PA 17837, USA}
\email{\href{mailto:peter.mcnamara@bucknell.edu}{peter.mcnamara@bucknell.edu}}


\begin{abstract}
We seek simple conditions on a pair of labeled posets that determine when the difference of their $(P,\om)$-partition enumerators is $F$-positive, i.e., positive in Gessel's fundamental basis.  This is a quasisymmetric analogue of the extensively studied problem of finding conditions on a pair of skew shapes that determine when the difference of their skew Schur functions is Schur-positive.  We determine necessary conditions and separate sufficient conditions for $F$-positivity, and show that a broad operation for combining posets preserves positivity properties.   We conclude with classes of posets for which we have conditions that are both necessary and sufficient.
\end{abstract}

\vspace*{-2mm}
\maketitle

\tableofcontents

\section{Introduction}\label{sec:intro} 
 
Symmetric functions play a major role in algebraic combinatorics, with positivity questions being a driving force behind their study.  Perhaps the most famous such positivity result is that the Littlewood--Richardson coefficients $c^\lambda_{\mu\nu}$ are non-negative.  Most importantly for us,  $c^\lambda_{\mu\nu}$ arises when we expand a skew Schur function in the Schur basis: 
\begin{equation}\label{equ:lr}
s_{\lambda/\mu} = \sum_\nu c^\lambda_{\mu\nu} s_\nu\,.  
\end{equation}
In particular, $s_{\lambda/\mu}$ is Schur-positive, meaning it has all non-negative coefficients when expanded in the Schur basis.  Going deeper into this positivity result, can we say that a given $s_{\lambda/\mu}$ is ``more Schur-positive'' than some other $s_{\alpha/\beta}$?  More precisely, when is $s_{\lambda/\mu} - s_{\alpha/\beta}$ Schur-positive?  Much attention has been given to this question, including \cite{BBR06, BaOr14,  DoPy07, FFLP05,  Kir04, KWvW08, LLT97, LPP07, McN14, McvW09b, McvW12, Oko97, RhSk06, RhSk10, TovW18}.  Two typical goals are to determine simple conditions on the shapes of $\lambda/\mu$ and $\alpha/\beta$ that tell us whether or not $s_{\lambda/\mu} - s_{\alpha/\beta}$ is Schur-positive, and to discover operations that can be performed on skew shapes that result in Schur-positive differences.

The 21st century has seen a surge in attention given to quasisymmetric functions.  Many of the quasisymmetric results have been motivated by related results in the symmetric setting, and our goal is to begin to address the quasisymmetric analogue of the question of Schur-positivity of $s_{\lambda/\mu} - s_{\alpha/\beta}$.  Given the large amount of work done on this symmetric question,  it is reasonable to expect there to be much scope in the quasisymmetric analogue, and we do not aim to be comprehensive.  One of our hopes is that this paper will serve as a starting point for further advances.  

The description of our quasisymmetric analogue begins with labeled posets $(P,\om)$, which are a broad but natural generalization of skew diagrams $\lambda/\mu$.  Under this extension, semistandard Young tableaux of shape $\lambda/\mu$ are generalized to $(P,\om)$-partitions, as first defined in \cite{StaThesis71,StaThesis}.  Moreover, the skew Schur function $s_{\lambda/\mu}$ is generalized to the generating function $\kpw$ for $(P,\om)$-partitions which, consistent with the literature, we will call the \emph{$(P,\om)$-partition enumerator}.  Example~\ref{exa:skew} below demonstrates that $s_{\lambda/\mu}$ is indeed a special case of $\kpw$.  As we will see, $\kpw$ is a quasisymmetric function;  see \cite{Ges84} and \cite[Sec.~ 7.19]{ec2} for further information about $\kpw$.

The ``right''  analogue of Schur-positivity is not as obvious.  Compelling quasisymmetric analogues of Schur functions include the quasisymmetric Schur functions of Haglund et al.\,\cite{HLMvW11,LMvW13} and the dual immaculate basis of Berg et al.\,\cite{BBSSZ14}. 
However, we choose to focus on positivity in Gessel's fundamental basis $F_\alpha$ for two main reasons.  The first is the appeal and simplicity of the expansion of $\kpw$ in the $F$-basis, which is Theorem~\ref{thm:kexpansion} below due to Gessel and Stanley; it can be viewed as a quasisymmetric analogue of \eqref{equ:lr}.  Secondly, the $F$-basis has shown its relevance in a wide variety of settings.  These include understanding equality of skew Schur functions \cite{McN14}, obtaining a positive expansion for genomic Schur functions \cite{Pec20} which are not Schur-positive in general, and in the realm of Macdonald polynomials \cite{CaMe18, Hag04, HHL05, HHLRU05}.  The $F_\alpha$ also have a representation-theoretic link: they are the quasisymmetric characteristics of the irreducible representations of the 0-Hecke algebra \cite{DHT02, DKLT96, KrTh97, Nor79}.  As one of the two original bases (the other being the monomial basis) from Gessel's introduction of quasisymmetric functions \cite{Ges84}, the $F$-basis has stood the test of time and has earned its name of ``the fundamental basis.''

Our goal, therefore, is to determine conditions on labeled posets $(P,\om)$ and $(Q,\tau)$ for $\kqt - \kpw$ to be $F$-positive, i.e., has only non-negative coefficients when expanded in the $F$-basis.  In this case, we write $(P,\om) \leqF (Q,\tau)$ and refer to $\leqF$ as the \emph{$F$-positivity order}.  An example of this relation is given in Figure~\ref{fig:first_example}.  
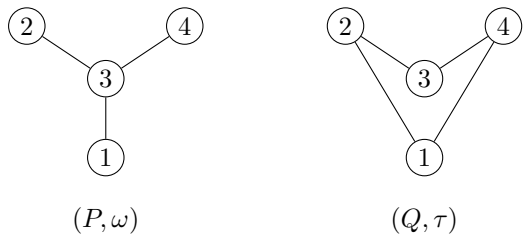
\begin{figure}
\begin{center}
\begin{tikzpicture}[scale=0.7]
\begin{scope}
\begin{scope}
\tikzstyle{every node}=[draw, shape=circle, inner sep=2pt]; 
\draw (0,0) node (a1) {1};
\draw (0,1.5) node (a3) {3};
\draw (-1.5,2.5) node (a2) {2};
\draw (1.5,2.5) node (a4) {4};
\draw (a1) -- (a3) -- (a4);
\draw (a3) -- (a2);
\end{scope}
\draw (0,-1.2) node {$(P,\om)$};
\end{scope}
\begin{scope}[xshift=40ex]
\begin{scope}
\tikzstyle{every node}=[draw, shape=circle, inner sep=2pt]; 
\draw (0,0) node (a1) {1};
\draw (0,1.5) node (a3) {3};
\draw (-1.5,2.5) node (a2) {2};
\draw (1.5,2.5) node (a4) {4};
\draw (a2) -- (a1) -- (a4) -- (a3);
\draw (a3) -- (a2);
\end{scope}
\draw (0,-1.2) node {$(Q,\tau)$};
\end{scope}
\end{tikzpicture}
\caption{An example of $(P,\om) \leqF (Q,\tau)$.}
\label{fig:first_example}
\end{center}
\end{figure}
We need to be a little careful since $\leqF$ is not a partial order as just introduced since there can be distinct labeled posets that have the same $(P,\om)$-partition enumerator.  The question of determining when $\kpw = \kqt$ was first considered in \cite{McWa14} with important further advances in each of \cite{Fer15, HaTs17,  LiWe20a, LiWe20b}.  So when we write $(P,\om) \leqF (Q,\tau)$, we are really considering $(P,\om)$ and $(Q,\tau)$ as representatives of their equivalence classes, where the equivalence relation is equality of $(P,\om)$-partition enumerators.  Prior to the present paper, the $F$-positivity order on labeled posets $(P,\om)$ was considered only in two papers.  Lam and Pylyavskyy's \cite{LaPy08} interest was in comparing products of $(P,\om)$-partition enumerators; in our language, this means that all their $(P,\om)$ have at least two connected components.  The second author \cite{McN14} previously considered skew-diagram labeled posets, as defined in Example~\ref{exa:skew} below. 

Since the corresponding Schur-positivity question from the first paragraph above remains wide open, it is unsurprising that we do not offer simple conditions that are both necessary and sufficient for $\kpw \leqF \kqt$.  Instead, after giving preliminaries in Section~\ref{sec:prelim}, we provide necessary conditions on $(P,\om)$ and $(Q,\tau)$ for $\kpw \leqF \kqt$ in Section~\ref{sec:nec}.  These are centered on the ideas of the jump sequence and Greene shape of a labeled poset.  It turns out that many of our results are optimally stated using variations of $F$-positivity such as $M$-positivity, where $M$ denotes the monomial basis, and containment of $F$-supports; see Subsection~\ref{sub:implications} for the full set of variations we use along with the relationships among them. In Section~\ref{sec:hasse}, we consider operations that can be performed on the Hasse diagram of a labeled poset $(P,\om)$ to produce a new labeled poset $(Q,\tau)$ that is larger in the $F$-positivity order.  We show that these operations are exactly those that result in the set of linear extensions of $(P,\om)$ being contained in that of $(Q,\tau)$. In Section~\ref{sec:ur}, we consider the operation of \emph{poset assembly}, which is called the \emph{Ur-operation} in \cite{BHK18+,BHK18}, and includes disjoint union, ordinal sum, and composition of posets as special cases.  We show that poset assembly does not preserve $F$-positivity in the most general possible sense, but we can get substantial results by strengthening our hypotheses.  In Section~\ref{sec:greene}, we consider two classes of posets for which we can give a simple condition that is both necessary and sufficient for $\kpw \leqF \kqt$ and for variations of this inequality.  Both classes involve a maximal chain with all other edges having one endpoint on this maximal chain. We conclude in Section~\ref{sec:conclusion} with a range of open problems.

\subsection*{Acknowledgements} 
We are grateful to Christophe Reutenauer for asking the second author about equality of $(P,\om)$-partition enumerators, of which this project is an outgrowth.  We thank the anonymous referee for helpful comments, including a reminder of the connection to the 0-Hecke algebra.  This paper is based on the first author's undergraduate honors thesis at Bucknell University where his research was funded by the Hoover Math Scholarship and the Department of Mathematics.  Portions of this paper were written while the second author was on sabbatical at Universit\'e de Bordeaux; he thanks LaBRI for its hospitality. We are grateful to Doriann Albertin, Jean-Christophe Aval and Hugo Mlodecki for pointing out the error in Section 6.2 of \cite{LeMc22}, which resulted in the correction \cite{LeMc22_correction}; the present version is a self-contained corrected version of \cite{LeMc22}.  Computations were performed using SageMath \cite{SageMath}.

\section{Preliminaries}\label{sec:prelim}

In this section, we give the necessary background on labeled posets, $(P,\om)$-partitions, $\kpw$, quasisymmetric functions, and $F$-positivity; more details can be found in \cite{Ges84} and \cite[Sec.~7.19]{ec2}.

\subsection{Labeled posets and \texorpdfstring{$(P,\om)$}{(P,w)}-partitions}  
Let $[n]$ denote the set $\{1, \ldots, n\}$ and let $1^n$ denote a sequence of $n$ copies of $1$.  The order relation on the poset $P$ will be denoted $\leq_P$, while $\leq$ will denote the usual order on the positive integers.  A \emph{labeling} of a poset $P$ is a bijection $\om : P \to [\vert P\vert ]$, where $\vert P\vert $ denotes the cardinality of $P$; all our posets will be finite.  A \emph{labeled poset} $(P,\om)$ is then a poset $P$ with an associated labeling $\om$.  

\begin{definition}\label{def:popartition}
For a labeled poset $(P,\om)$, a \emph{$(P,\om)$-partition} is a map $\popf$ from $P$ to the positive integers satisfying the following two conditions:
\begin{itemize}
\item if $a <_P b$, then $\popf(a) \leq \popf(b)$, i.e., $\popf$ is order-preserving;
\item if $a <_P b$ and $\om(a) > \om(b)$, then $\popf(a) < \popf(b)$.
\end{itemize}
\end{definition}
In other words, a $(P,\om)$-partition is an order-preserving map from $P$ to the positive integers with certain strictness conditions determined by $\om$.  
Examples of $(P,\om)$-partitions are given in Figure~\ref{fig:popartitions}, where the images under $\popf$ are written in bold next to the nodes.  The meaning of the double edges in the figure follows from the following observation about Definition~\ref{def:popartition}.  For $a, b \in P$, we say that $a$ is \emph{covered} by $b$ in $P$, denoted $a \pcovered b$, if $a <_P b$ and there does not exist $c$ in $P$ such that $a <_P c <_P b$.  Note that a definition equivalent to Definition~\ref{def:popartition} is obtained by replacing both appearances of the relation $a <_P b$ with the relation $a \pcovered b$.  In other words, we require that $\popf$ be order-preserving along the edges of the Hasse diagram of $P$, with $\popf(a) < \popf(b)$ when the edge $a \pcovered b$ satisfies $\om(a) > \om(b)$.  With this in mind, we will consider those edges $a \pcovered b$ with $\om(a) > \om(b)$ as \emph{strict edges} and we will represent them in Hasse diagrams by double lines.  Similarly, edges $a \pcovered b$ with $\om(a) < \om(b)$ will be called \emph{weak edges} and will be represented by single lines.

From the point-of-view of $(P,\om)$-partitions, the labeling $\om$ only determines which edges are strict and which are weak.  Therefore, many of our figures from this point on will not show the labeling $\om$, but instead show some collection of strict and weak edges determined by an underlying $\om$.  Furthermore, it will make many of our explanations simpler if we think of $\om$ as an assignment of strict and weak edges, rather than as a labeling of the elements of $P$, especially in the later sections. For example, we say that labeled posets $(P,\om)$ and $(Q,\tau)$ are \emph{isomorphic}, written $(P,\om) \cong (Q,\tau)$, if there exists a poset isomorphism from $P$ to $Q$ that sends strict (respectively weak) edges to strict (resp.\ weak) edges.  We will be careful to refer to the underlying labels when necessary.  
\begin{figure}[htbp]
\begin{center}
\begin{tikzpicture}[scale=0.7]
\begin{scope}
\tikzstyle{every node}=[shape=circle, inner sep=2pt]; 
\draw (0,0) node[draw] (a1) {1} +(0.6,-0.3) node {\textcolor{blue}{\textbf{3}}};
\draw (0,1.5) node[draw] (a3) {3} +(0.0,0.65) node {\textcolor{blue}{\textbf{3}}};
\draw (-1.5,2.5) node[draw] (a2) {2} +(-0.6,-0.3) node {\textcolor{blue}{\textbf{4}}};
\draw (1.5,2.5) node[draw] (a4) {4} +(0.6,-0.3) node {\textcolor{blue}{\textbf{7}}};
\draw (a1) -- (a3) -- (a4);
\draw[double distance=2pt] (a3) -- (a2);
\draw (0,-1.5) node {(a)};
\end{scope}
\begin{scope}[xshift=40ex]
\tikzstyle{every node}=[shape=circle, inner sep=2pt]; 
\draw (0,0) node[draw] (a1) {1} +(0.6,-0.3) node {\textcolor{blue}{\textbf{4}}};
\draw (0,1.5) node[draw] (a3) {3} +(0.0,0.65) node {\textcolor{blue}{\textbf{3}}};
\draw (-1.5,2.5) node[draw] (a2) {2} +(-0.6,-0.3) node {\textcolor{blue}{\textbf{4}}};
\draw (1.5,2.5) node[draw] (a4) {4} +(0.6,-0.3) node {\textcolor{blue}{\textbf{7}}};
\draw (a2) -- (a1) -- (a4) -- (a3);
\draw[double distance=2pt] (a3) -- (a2);
\draw (0,-1.5) node {(b)};
\end{scope}
\end{tikzpicture}
\caption{Two examples of $(P,\om)$-partitions.}
\label{fig:popartitions}
\end{center}
\end{figure}
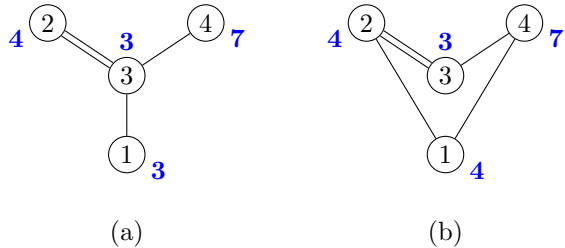

If all the edges are weak, then $P$ is said to be \emph{naturally labeled}.  In this case, we typically omit reference to the labeling and so a $(P,\om)$-partition is then traditionally called a $P$-partition (although sometimes the term ``$P$-partition'' is used informally as an abbreviation for ``$(P,\om)$-partition,'' as in the title of this paper).  Note that if $P$ is a naturally labeled chain, then a $P$-partition gives a partition of an integer, while if $P$ is an antichain, a $P$-partition gives a composition of an integer; interpolating between partitions and compositions was a motivation for Stanley's definition of $(P,\om)$-partitions \cite{StaThesis71, StaThesis}.

\subsection{The \texorpdfstring{$(P,\om)$}{(P,w)}-partition enumerator}

Using $x$ to denote the sequence of variables $x_1, x_2, \ldots$, we can now define our main object of study.

\begin{definition}\label{def:kpo}
For a labeled poset $(P,\om)$, we define the \emph{$(P,\om)$-partition enumerator} $\kpw = \kpw(x)$ by
\[
\kpw(x) = \sum_{(P,\om)\textnormal{-partition }f} x_1^{\vert \popf^{-1}(1)\vert } x_2^{\vert \popf^{-1}(2)\vert } \cdots,
\]
where the sum is over all $(P,\om)$-partitions $\popf$.
\end{definition}

For example, the $(P,\om)$-partition in Figure~\ref{fig:popartitions}(a) would contribute the monomial $x_3^2 x_4 x_7$ to its $\kpw$.  As another simple example, if $(P,\om)$ is a naturally labeled chain with three elements, then $\kpw = \sum_{i \leq j \leq k} x_i x_j x_k$.

\begin{example}\label{exa:skew}
Given a skew diagram $\lambda/\mu$ in French notation with $n$ cells, bijectively label the cells with the numbers $[n]$ in any way that makes the labels increase down columns and from left-to-right along rows, as in Figure~\ref{fig:skewschur}(a).   Rotating the result 45$^\circ$ in a counter-clockwise direction and replacing the cells by nodes as in Figure~\ref{fig:skewschur}(b), we get a corresponding labeled poset which we denote by $(\sd{\lambda/\mu}, \om_{\lambda/\mu})$ and call a \emph{skew-diagram labeled poset}.  Under this construction, we see that a $(\sd{\lambda/\mu}, \om_{\lambda/\mu})$-partition corresponds exactly to a semistandard Young tableau of shape $\lambda/\mu$. Therefore $\gen{\sd{\lambda/\mu}, \om_{\lambda/\mu}}$ is exactly the skew Schur function $s_{\lambda/\mu}$,
justifying the claim in the introduction that the study of $\kqt -\kpw$  is a generalization of the study of $s_{\lambda/\mu} - s_{\alpha/\beta}$.
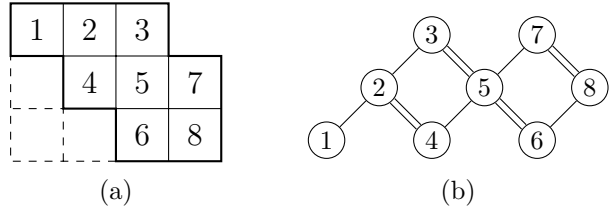
\begin{figure}[htbp]
\begin{center}
\begin{tikzpicture}[scale=0.7]
\begin{scope}[yshift=-0.4cm]
\draw[thick] (0,3) -- (3,3) -- (3,2) -- (4,2) -- (4,0) -- (2,0) -- (2,1) -- (1,1) -- (1,2) -- (0,2) -- cycle;
\draw (1,2) -- (3,2) 
(2,1) -- (4,1) 
(1,2) -- (1,3) 
(2,1) -- (2,3) 
(3,0) -- (3,2);
\draw[dashed] (2,0) -- (0,0) -- (0,2)
(1,0) -- (1,1) -- (0,1);
\begin{scope}[font = \Large]
\draw (0.5,2.5) node {1};
\draw (1.5,2.5) node {2};
\draw (2.5,2.5) node {3};
\draw (1.5,1.5) node {4};
\draw (2.5,1.5) node {5};
\draw (2.5,0.5) node {6};
\draw (3.5,1.5) node {7};
\draw (3.5,0.5) node {8};
\end{scope}
\end{scope}
\draw (2,-1) node {(a)};
\begin{scope}[xshift=6cm]
\begin{scope} 
\tikzstyle{every node}=[draw, shape=circle, inner sep=2pt]; 
\draw (0,0) node (a1) {1};
\draw (1,1) node (a2) {2};
\draw (2,2) node (a3) {3};
\draw (2,0) node (a4) {4};
\draw (3,1) node (a5) {5};
\draw (4,0) node (a6) {6};
\draw (4,2) node (a7) {7};
\draw (5,1) node (a8) {8};
\draw[double distance=2pt] (a4) -- (a2)
(a6) -- (a5) -- (a3)
(a8) -- (a7);
\draw (a1) -- (a2) --(a3)
(a4) -- (a5) -- (a7)
(a6) -- (a8); 
\end{scope}
\draw (2.5,-1) node {(b)};
\end{scope}
\end{tikzpicture}
\caption{The skew diagram $443/21$ and a corresponding labeled poset.}
\label{fig:skewschur}
\end{center}
\end{figure}
\end{example}

\subsection{Quasisymmetric functions}

It follows directly from the definition of quasisymmetric functions below that $\kpw$ is quasisymmetric.  In fact, $\kpw$ served as motivation for Gessel's original definition \cite{Ges84} of quasisymmetric functions.   

We will make use of both of the classical bases for the vector space of quasisymmetric functions.  If $\alpha = (\alpha_1, \alpha_2, \ldots, \alpha_k)$ is a composition of $n$, then we define the \emph{monomial quasisymmetric function} $M_\alpha$ by
\[
M_\alpha = \sum_{i_1 < i_2 < \ldots < i_k} x_{i_1}^{\alpha_1} x_{i_2}^{\alpha_2} \cdots x_{i_k}^{\alpha_k}.
\]
This monomial basis is natural enough that we can see directly from Definition~\ref{def:kpo} how to expand $\kpw$ in this basis.

\begin{proposition}\label{pro:mexpansion}
For a labeled poset $(P,\om)$ and a composition $\alpha = (\alpha_1, \ldots, \alpha_k)$, the coefficient of $M_\alpha$ in $\kpw$ will be the number of $(P,\om)$-partitions $\popf$ such that $\vert \popf^{-1}(1)\vert  = \alpha_1$,\ldots, $\vert \popf^{-1}(k)\vert  = \alpha_k$.  
\end{proposition}

As we know, compositions $\alpha = (\alpha_1, \alpha_2, \ldots, \alpha_k)$ of $n$ are in bijection with subsets of $[n-1]$, and let $S(\alpha)$ denote the set
$\{\alpha_1, \alpha_1+\alpha_2, \ldots, \alpha_1+\alpha_2 + \cdots+ \alpha_{k-1}\}$.  It will be helpful to sometimes denote $M_\alpha$ by $M_{S(\alpha), n}$.  Notice that these two notations are distinguished by the latter one including the subscript $n$; the subscript is helpful since $S(\alpha)$ does not uniquely determine $n$.  

The second classical basis is composed of the \emph{fundamental quasisymmetric functions} $F_\alpha$ defined by 
\begin{equation}\label{equ:F2M}
F_\alpha = F_{S(\alpha), n} = \sum_{S(\alpha) \subseteq T  \subseteq [n-1]} M_{T,n}\,.
\end{equation}
The relevance of this latter basis to $\kpw$ is due to Theorem~\ref{thm:kexpansion} below, which first appeared in \cite{StaThesis71,StaThesis} and, in the language of quasisymmetric functions, in \cite{Ges84}.  

Every permutation $\pi \in S_n$ has a descent set $\des(\pi)$ given by $\{i \in [n-1] : \pi(i) > \pi(i+1)\}$, and we will call the corresponding composition of $n$ the \emph{descent composition} of $\pi$.  Let $\L(P,\om)$ denote the set of all linear extensions of $P$, regarded as permutations of the $\om$-labels of $P$.  For example, for the labeled poset in Figure~\ref{fig:popartitions}(b), $\L(P,\om) = \{1324, 1342, 3124, 3142\}$.  

\begin{theorem}[\cite{Ges84,StaThesis71,StaThesis}]\label{thm:kexpansion}
Let $(P,\om)$ be a labeled poset with $\vert P\vert =n$.  Then
\[
\kpw = \sum_{\pi \in \L(P,\om)} F_{\des(\pi), n}\,.
\]
\end{theorem}

\begin{example}\label{exa:kexpansion}
The labeled poset $(P,\om)$ of Figure~\ref{fig:popartitions}(a) has $\L(P,\om) = \{1324, 1342\}$ and hence
\begin{align*}
\gen{P, \om} &= F_{\{2\},4} + F_{\{3\},4} \\
&=  F_{22} + F_{31} \\ 
&=  M_{\{2\},4} + M_{\{3\},4} + M_{\{1,2\},4} + M_{\{1,3\},4} + 2M_{\{2,3\},4} + 2M_{\{1,2,3\},4}\\
&=  M_{22} + M_{31} + M_{112}  + M_{121} + 2M_{211} + 2M_{1111}.
\end{align*}
\end{example}

\subsection{Positivity implications}\label{sub:implications}

Although our introduction focused on $\kqt - \kpw$ being $F$-positive, there are some stronger and weaker conditions that are important for us.  First, referring to Theorem~\ref{thm:kexpansion}, a sufficient condition for $(P,\om) \leqF (Q,\tau)$ is that $\L(P,\om) \subseteq \L(Q,\tau)$.  We will investigate this latter condition in detail in Section~\ref{sec:hasse}, and some of our results in Section~\ref{sec:ur} will need us to strengthen our hypotheses by using this linear extension containment.  

As follows from~\eqref{equ:F2M}, a weaker condition than $\kqt - \kpw$ being $F$-positive is that $\kqt - \kpw$ is $M$-positive, which we denote by $(P,\om) \leqM (Q,\tau)$.  Moreover, the \emph{$F$-support} of $(P,\om)$, denoted $\Fsupp(P,\om)$, is the set of compositions $\alpha$ such that $F_\alpha$ appears with nonzero (and hence positive) coefficient in the $F$-expansion of $\kpw$.  We define the \emph{$M$-support} $\Msupp(P,\om)$ analogously.  For example, the $M$-support of the poset from Example~\ref{exa:kexpansion} is $\{22, 31, 112, 121, 211, 1111\}$.   By definition of these supports and \eqref{equ:F2M}, we get the following set of implications.
\begin{equation}\label{equ:implications}
\begin{tikzpicture}[scale=0.8]
\tikzstyle{every node}=[draw, inner sep=3pt]; 
\node at (0,3) {$\L(P,\om) \subseteq \L(Q,\tau)$};
\node at (0,1.5) {$(P,\om) \leqF (Q,\tau)$};
\node at (0,0) {$(P,\om) \leqM (Q,\tau)$};
\node at (5.5,1.5) {$\Fsupp(P,\om) \subseteq \Fsupp(Q,\tau)$};
\node at (5.5,0) {$\Msupp(P,\om) \subseteq \Msupp(Q,\tau)$};
\tikzstyle{every node}=[]; 
\node at (2.25,0) {$\Rightarrow$};
\node at (2.25,1.5) {$\Rightarrow$};
\node at (0,0.75) {$\Downarrow$};
\node at (5.5,0.75) {$\Downarrow$};
\node at (0,2.25) {$\Downarrow$};
\end{tikzpicture}
\end{equation}
Some of our necessary conditions for $(P,\om) \leqF (Q,\tau)$ in the next section will merely require us to assume that $\Msupp(P,\om) \subseteq \Msupp(Q,\tau)$; we refer to the latter condition as ``$M$-support containment.''

\subsection{Involutions on labeled posets}

For the last preliminary, we consider some natural involutions we can perform on a labeled poset $(P,\om)$. Subsets of the set of these involutions appear previously in, for example, \cite{Ehr96,Ges90,LaPy08,MalThesis, MaRe95, MaRe98} and \cite[Exer.~7.94(a)]{ec2}.  First, we can switch strict and weak edges, denoting the result $\switch{P,\om}$. Secondly, we can rotate the labeled poset 180$^\circ$, preserving strictness and weakness of edges; we denote the resulting labeled poset $\rotatep{P,\om}$.  Observe that these so-called \emph{bar} and \emph{star involutions} commute; an example is given in Figure~\ref{fig:involutions}.    

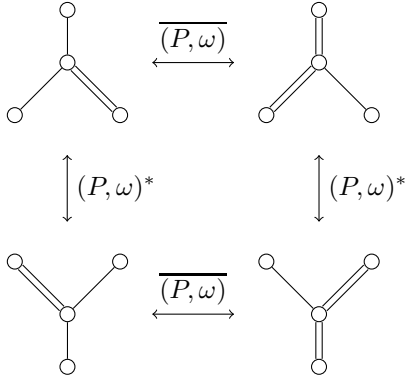
\begin{figure}[htbp]
\begin{center}
\begin{tikzpicture}[scale=1.0]
\matrix [row sep=15mm, column sep=15mm]  {
\node(NW) {
\begin{tikzpicture}[scale=0.7]
\tikzstyle{every node}=[draw, shape=circle, inner sep=2pt]; 
\draw (0,0) node (a1) {};
\draw (2,0) node (a2) {};
\draw (1,1) node (a3) {};
\draw (1,2) node (a4) {};
\draw (a1) -- (a3)
(a3) -- (a4);
\draw[double distance=2pt] (a2) -- (a3);
\end{tikzpicture}
};
&
\node(NE) {
\begin{tikzpicture}[scale=0.7]
\tikzstyle{every node}=[draw, shape=circle, inner sep=2pt]; 
\draw (0,0) node (b1) {};
\draw (2,0) node (b2) {};
\draw (1,1) node (b3) {};
\draw (1,2) node (b4) {};
\draw[double distance=2pt] (b1) -- (b3)
(b3) -- (b4);
\draw (b2) -- (b3); 
\end{tikzpicture}
};
\\
\node(SW) {
\begin{tikzpicture}[scale=0.7]
\tikzstyle{every node}=[draw, shape=circle, inner sep=2pt]; 
\draw (1,0) node (c1) {};
\draw (1,1) node (c2) {};
\draw (0,2) node (c3) {};
\draw (2,2) node (c4) {};
\draw (c1) -- (c2)
(c2) -- (c4);
\draw[double distance=2pt](c2) -- (c3);
\end{tikzpicture}
};
&
\node(SE) {
\begin{tikzpicture}[scale=0.7]
\tikzstyle{every node}=[draw, shape=circle, inner sep=2pt]; 
\draw (1,0) node (d1) {};
\draw (1,1) node (d2) {};
\draw (0,2) node (d3) {};
\draw (2,2) node (d4) {};
\draw (d2) -- (d3);
\draw[double distance=2pt](d1) -- (d2)
(d2) -- (d4);
\end{tikzpicture}
};
\\ }; 
\draw[<->,shorten >=2mm, shorten <= 2mm] (NW) -- (NE) node [midway,auto] {$\switch{P,\om}$};
\draw[<->,shorten >=3mm, shorten <=3mm] (NW) -- (SW) node [midway,auto] {$\rotatep{P,\om}$};
\draw[<->,shorten >=3mm, shorten <=3mm] (NE) -- (SE) node [midway,auto] {$\rotatep{P,\om}$};
\draw[<->,shorten >=2mm, shorten <= 2mm] (SW) -- (SE) node [midway,auto] {$\switch{P,\om}$};
\end{tikzpicture}
\caption{The bar and star involutions.}
\label{fig:involutions}
\end{center}
\end{figure}

Although in subsequent sections we will prefer to view these bar and star involutions as described above, it will be helpful for Lemma~\ref{lem:involutions} below to formulate them in terms of their effect on the $\om$-labels.  
We see that 
\[
\switch{P,\om} \cong (P,\overline{\om}),
\]
where $\switchx{\om}$ is defined by 
\[
\switchx{\om}(a) = \vert P\vert +1-\om(a)
\]
for all $a \in P$.  Let $\rotatex{P}$ denote the 180$^\circ$ rotation (dual) of $P$, and let $\rotatex{a}$ denote the image of $a\in P$ under this rotation,  Then
\[
\rotatep{P,\om} \cong (\rotatex{P},\rotatex{\om}),
\]
where $\rotatex{\om}$ is defined by 
\[
\rotatex{\om}(\rotatex{a}) = \vert P\vert +1-\om(a).
\]

Given Theorem~\ref{thm:kexpansion}, we can determine the effect of these involutions on $\kpw$ by examining their effect on the linear extensions of $(P,\om)$.  A proof can be found in \cite{McWa14}.

\begin{lemma}\label{lem:involutions}  Let $(P,\om)$ be a labeled poset.  Then we have:
\begin{enumerate}
\item the descent sets of the linear extensions of $\switch{P,\om}$ are the complements of the descent sets of the linear extensions of $(P,\om)$;
\item the descent compositions of the linear extensions of $\rotatep{P,\om}$ are the reverses of the descent compositions of the linear extensions of $(P,\om)$.  
\end{enumerate}
\end{lemma}

The usefulness of these involutions stems from the following consequence of Theorem~\ref{thm:kexpansion} and Lemma~\ref{lem:involutions}

\begin{proposition}\label{pro:fourfold}
For labeled posets $(P,\om)$ and $(Q,\tau)$, the following are equivalent:
\begin{enumerate}
\item $(P,\om) \leqF (Q,\tau)$;\medskip
\item $\switch{P,\om} \leqF \switch{Q,\tau}$;\medskip
\item $\rotatep{P,\om} \leqF \rotatep{Q,\tau}$;\medskip
\item $\switchx{\rotatep{P,\om}} \leqF \switchx{\rotatep{Q,\tau}}$.
\end{enumerate}
In addition, $F$-support containment is also preserved by the bar and star involutions.
\end{proposition}

In particular, any condition that is necessary (resp.\ sufficient) for the inequality of (a) will also be necessary (resp.\ sufficient) for the inequalities of (b), (c), and (d).

The equivalence of (a) and (b) in Proposition~\ref{pro:fourfold} does not hold if we replace ``$\leqF$'' with ``$\leqM$'' in both inequalities, nor if we replace ``$F$-support containment'' with ``$M$-support containment.''  For example, let $(P,\om)$ be a 2-element chain with a strict edge and let $(Q,\tau)$ be a 2-element chain with a weak edge.  However, we do get the following result.

\begin{proposition}\label{pro:twofold}
For labeled posets $(P,\om)$ and $(Q,\tau)$, we have $(P,\om) \leq_M (Q,\tau)$ if and only if $\rotatep{P,\om} \leqM \rotatep{Q,\tau}$.  In addition, $M$-support containment is also preserved under the star involution.
\end{proposition}

\begin{proof}
Referring to Proposition~\ref{pro:mexpansion}, let $\popf$ be a $(P,\om)$-partition such that $\vert \popf^{-1}(1)\vert  = \alpha_1$,\ldots, $\vert \popf^{-1}(k)\vert  = \alpha_k$.  Each such $(P,\om)$-partition maps bijectively to a $\rotatep{P,\om}$-partition $g$ defined by $g(\rotatex{a}) = k+1-f(a)$, where $\rotatex{a} \in \rotatex{P}$ is the image of $a \in P$ under the 180$^\circ$ rotation.  The key property of $g$ is that $\vert g^{-1}(1)\vert  = \alpha_k$,\ldots, $\vert g^{-1}(k)\vert  = \alpha_1$\,.  By Proposition~\ref{pro:mexpansion} and letting $\rotatex{\alpha}$ denote the reversal of $\alpha$, the coefficient of any $M_\alpha$ in $\kpw$ is thus equal to the coefficient of $M_{\rotatex{\alpha}} \in \genx{\rotatep{P,\om}}$, from which the result follows.
\end{proof}

\section{Necessary conditions}\label{sec:nec}

Our ultimate goal is to determine the full set of relations of the form $(P,\om) \leqF (Q,\tau)$ in terms of simple conditions on $(P,\om)$ or $(Q,\tau)$, along with similar results for our other ways of comparing $(P,\om)$ and  $(Q,\tau)$ from~\eqref{equ:implications}.  To do so, we would need to know which inequalities do \emph{not} hold, and necessary conditions give us a way to find such instances.  In the interest of generality, we will state each of our necessary conditions using the weakest hypothesis from~\eqref{equ:implications} needed.

\subsection{Quick observations}

First, it is clearly the case that if $\Msupp(P,\om) \subseteq \Msupp(Q,\tau)$ then $\vert P\vert =\vert Q\vert $.  Our second observation is almost as simple.

\begin{proposition}
If labeled posets $(P,\om)$ and $(Q,\tau)$ satisfy $(P,\om) \leqM (Q,\tau)$, then $\vert \L(P,\om)\vert  \leq \vert \L(Q,\tau)\vert $.
\end{proposition}


\begin{proof}
Let $n=\vert P\vert $.  By Proposition~\ref{pro:mexpansion}, the coefficient of $M_{1^n}$ in $\kpw$ is the number of $(P,\om)$-partitions $\popf$ that are bijections to the set $[n]$.  Considering $\popf$ as a way to order the elements of $(P,\om)$, the coefficient of $M_{1^n}$ is thus the number of linear extensions of $(P,\om)$, and the result follows.  
\end{proof}

The next result answers the question of whether a naturally labeled poset can be less than a non-naturally labeled poset.

\begin{proposition}\label{pro:natural_preserved}
If labeled posets $(P,\om)$ and $(Q,\tau)$ satisfy \newline $\Msupp(P,\om) \subseteq \Msupp(Q,\tau)$ and $(P,\om)$ has only weak edges, then so does $(Q,\tau)$.
\end{proposition}

\begin{proof}
Observe that $(P,\om)$ has only weak edges if and only if there is a $(P,\om)$-partition that maps all elements to the number 1, meaning that $(\vert P\vert)$ is in the $M$-support of $(P,\om)$.  Thus $\vert P\vert $ is in the $M$-support of $(Q,\tau)$, and $(Q,\tau)$ has only weak edges.
\end{proof}

Of course, Proposition~\ref{pro:natural_preserved} remains true if we replace the $M$-support containment hypothesis with the stronger condition that $\Fsupp(P,\om) \subseteq \Fsupp(Q,\tau)$.  Then applying the bar involution and Proposition~\ref{pro:fourfold}, we get the following corollary: if $\Fsupp(P,\om) \subseteq \Fsupp(Q,\tau)$ and $(P,\om)$ has only strict edges, then so does $(Q,\tau)$.

\subsection{Orderings on the jump}\label{subsec:jump}

A more sophisticated tool for deriving necessary conditions is the \emph{jump sequence} of a labeled poset, as defined in \cite{McWa14}.  The jump of an element of a labeled poset is very similar to the notion of maximum descent distance from \cite{AsBe20}.

\begin{definition}
We define the \emph{jump} of an element $b$ of a labeled poset $(P,\om)$ 
by considering the number of strict edges on each of the saturated chains from $b$ down to a minimal element of $P$, and taking the maximum such number.
The \emph{jump sequence} of $(P,\om)$, denoted $\jump(P,\om)$, is 
\[
\jump(P,\om) = (j_0, \ldots, j_k),
\]
where $j_i$ is the number of elements with jump $i$, and $k$ is the maximum jump of an element of $(P,\om)$.  
\end{definition}

For example, the posets of Figure~\ref{fig:jump} both have jump sequence $(3,1)$.  There is an alternative way to think of the jump sequence in terms of the $M$-support, as we now explain.  The dominance order on partitions of $n$ is well known, and its definition extends directly to give the dominance order on compositions, denoted $\domleq$\,.
 
\begin{lemma}\label{lem:jump}
Under dominance order, the maximum element of both the $M$-support and the $F$-support of a labeled poset $(P,\om)$ is the jump sequence of $(P,\om)$.
\end{lemma}

\begin{proof}
Consider the \emph{greedy} $(P,\om)$-partition $g$, defined in the following way.  Choose $g$ so that $\vert g^{-1}(1)\vert $ is as large as possible and, from the remaining elements of $P$, $\vert g^{-1}(2)\vert $ is as large as possible, and so on. We see that
\[
(\vert g^{-1}(1)\vert , \ldots, \vert g^{-1}(k)\vert ) = \jump(P,\om)
\]
where $k$ is the maximum jump of an element of $(P,\om)$.  By Proposition~\ref{pro:mexpansion}, $\jump(P,\om)$ is thus in the $M$-support of $(P,\om)$ and, by construction of $g$, is the maximum such element under dominance order.

The result for the $F$-support can be shown directly in terms of linear extensions, but it follows quickly from what we just proved for the $M$-support.  Indeed, by~\eqref{equ:F2M}, we know that $F_\alpha$ takes the form $M_\alpha + \sum_\beta M_\beta$, where every $\beta$ is strictly less than $\alpha$ in dominance order.  Therefore the maximum elements in the $F$-support and $M$-support are equal.\end{proof}

The following result gives our first main necessary conditions.

\begin{corollary}\label{cor:jump}
If labeled posets $(P,\om)$ and $(Q,\tau)$ satisfy $\Fsupp(P,\om) \subseteq \Fsupp(Q,\tau)$ or even just $\Msupp(P,\om) \subseteq \Msupp(Q,\tau)$, then $\jump(P,\om) \domleq \jump(Q,\tau)$.  
\end{corollary}

Because of the second sentence in Proposition~\ref{pro:fourfold}, we can get three additional necessary conditions for $\Fsupp(P,\om) \subseteq \Fsupp(Q,\tau)$ by applying the bar and star involutions and then Corollary~\ref{cor:jump}:
\begin{enumerate}
\item $\jump(\switch{P,\om}) \domleq \jump(\switch{Q,\tau})$;\medskip
\item $\jump(\rotatep{P,\om}) \domleq \jump(\rotatep{Q,\tau})$;\medskip
\item $\jump(\switchx{\rotatep{P,\om}}) \domleq \jump(\switchx{\rotatep{Q,\tau}})$.
\end{enumerate}
By Proposition~\ref{pro:twofold}, the inequality in (b) is also a necessary condition for \newline $\Msupp(P,\om) \subseteq \Msupp(Q,\tau)$.

\begin{example}\label{exa:barjump}
Consider the labeled posets $(P,\om)$ and $(Q,\tau)$ in Figure~\ref{fig:jump}.%
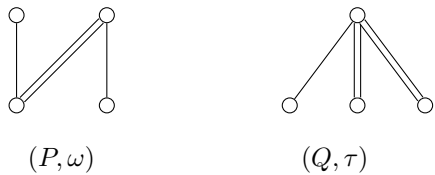
\begin{figure}
\begin{center}
\begin{tikzpicture}[scale=0.6]
\begin{scope}
\begin{scope}
\tikzstyle{every node}=[draw, shape=circle, inner sep=2pt]; 
\draw (0,0) node (a3) {};
\draw (0,2) node (a4) {};
\draw (2,0) node (a1) {};
\draw (2,2) node (a2) {};
\draw (a3) -- (a4);
\draw (a2) -- (a1);
\draw[double distance=2pt] (a3) -- (a2);
\end{scope}
\draw (1,-1.2) node {$(P,\om)$};
\end{scope}
\begin{scope}[xshift=40ex]
\begin{scope}
\tikzstyle{every node}=[draw, shape=circle, inner sep=2pt]; 
\draw (0,0) node (a1) {};
\draw (1.5,2) node (a2) {};
\draw (1.5,0) node (a3) {};
\draw (3,0) node (a4) {};
\draw (a1) -- (a2);
\draw[double distance=2pt] (a3) -- (a2);
\draw[double distance=2pt] (a4) -- (a2);
\end{scope}
\draw (1,-1.2) node {$(Q,\tau)$};
\end{scope}
\end{tikzpicture}
\caption{$(P,\om)$ and  $(Q,\tau)$ have the same jump sequence, but different jump sequences after applying either the bar or star involution.}
\label{fig:jump}
\end{center}
\end{figure}
Even though their jump sequences are equal, we have $\jump(\rotatep{P,\om}) = (3,1)$ whereas $\jump(\rotatep{Q,\tau}) = (2,2)$, implying $\Fsupp(\rotatep{P,\om}) \not\subseteq \Fsupp(\rotatep{Q,\tau})$ or, equivalently, $\Fsupp(P,\om) \not\subseteq \Fsupp(Q,\tau)$.  Moreover, $(P,\om)$ and $(Q,\tau)$ are actually incomparable under $F$-support containment since $\jump(\switch{P,\om}) = (2,2)$ whereas $\jump(\switch{Q,\tau}) = (3,1)$, so $\Fsupp(Q,\tau) \not\subseteq \Fsupp(P,\om)$.

Using an analogous approach, we get that $\Msupp(\rotatep{P,\om}) \not\subseteq \Msupp(\rotatep{Q,\tau})$, but then we can check computationally that $\Msupp(Q,\tau) \subseteq \Msupp(P,\om)$.  This is one example of the different behaviors of $F$-support and $M$-support containments.
\end{example}

As another example of the utility of the extra necessary conditions, for a naturally labeled poset $(P,\om)$, we have $\jump(P,\om)=(\vert P\vert )$, which does not give useful information whereas looking at $\jump(\switch{P,\om})$ might allow us to draw a conclusion.

Corollary~\ref{cor:jump} gives us the following necessary condition about maximal chains.

\begin{corollary}\label{cor:height}
If labeled posets $(P,\om)$ and $(Q,\tau)$ satisfy \newline $\Fsupp(P,\om) \subseteq \Fsupp(Q,\tau)$, then the maximum number of strict (resp.~weak) edges in a maximal chain in $(P,\om)$ must be greater than or equal to the maximum number of strict (resp.~weak) edges in a maximal chain in $(Q,\tau)$.  
\end{corollary}

\begin{proof}
By Corollary~\ref{cor:jump}, we know that $\jump(P,\om) \domleq \jump(Q,\tau)$, in which case the length of $\jump(P,\om)$ must be greater than or equal to the length of $\jump(Q,\tau)$.  Observe that the length of the jump sequence of a labeled poset $(P,\om)$ equals 1 plus the maximum number of strict edges on a saturated chain from a maximal element down to a minimal element of $P$. 

For the statement about weak edges, apply the bar involution and Proposition~\ref{pro:fourfold}.
\end{proof}

\!If we use the weaker hypothesis in Corollary~\ref{cor:height} that $\Msupp(P,\om) \subseteq \Msupp(Q,\tau)$, the first paragraph of the proof still works and we can draw the same conclusion about the number of strict edges.  The $(P,\om)$ and $(Q,\tau)$ mentioned just before Proposition~\ref{pro:twofold} give a counterexample to the statement about weak edges holding in the case of $M$-support containment.

Inspired by the unresolved \cite[Question~4.4]{McWa14}, we are intrigued by the potential
for a positive answer to the following question.

\begin{question}\label{que:longestchain}
Is Corollary~\ref{cor:height} true when we remove the requirement that the edges be strict (resp.\ weak) and instead consider all the edges on a maximal chain?
\end{question}

The answer to Question~\ref{que:longestchain} is ``yes'' when $\vert P\vert  \leq 6$, and in the case of skew-diagram labeled posets, as follows from \cite[Theorem~4.1]{McN14}.

\subsection{Convex subposets and Greene shape}\label{sub:convex}

As we saw in the proof of \linebreak Lemma~\ref{lem:jump}, we obtain the jump sequence by being greedy starting at the bottom of the poset.  This subsection leads off with the idea of being ``even greedier'' by giving ourselves more freedom in where we start.  The resulting necessary condition becomes particularly nice when we apply it to naturally labeled posets.

Recall that a \emph{convex subposet} $S$ of $P$ is a subposet of $P$ such that if $x,y,z \in P$ with $x <_P y <_P z$, then $x, z \in S$ implies $y \in S$.  A convex subposet of $(P,\om)$ inherits its designation of strict and weak edges from $(P,\om)$, and let us say that a convex subposet is \emph{weak} (resp.\ \emph{strict}) if all its edges are weak (resp.\ strict).  

\begin{proposition}\label{pro:convex}
Suppose labeled posets $(P,\om)$ and $(Q,\tau)$ satisfy $\Msupp(P,\om) \subseteq \Msupp(Q,\tau)$.  Then, for all $i$, the maximum cardinality of a union of $i$ weak convex subposets of $(P,\om)$ must be less than or equal to that of $(Q,\tau)$.
\end{proposition}

\begin{proof}  Observe that in a labeled poset $(P,\om)$, a subposet $S$ forms a weak convex subposet if and only if its elements can all have the same image in some $(P,\om)$-partition.  Therefore, the maximum cardinality of a union of $i$ weak convex subposets of $(P,\om)$ is exactly the largest sum $\alpha_{j_1}+\cdots+\alpha_{j_i}$ of $i$ distinct elements of  $\alpha$, among all $\alpha \in \Msupp(P,\om)$.  The result follows since $\Msupp(P,\om) \subseteq \Msupp(Q,\tau)$.
\end{proof}

As usual, if we use the stronger hypothesis that $\Fsupp(P,\om) \subseteq \Fsupp(Q,\tau)$ and the bar involution, we can draw the analogous conclusion about strict convex subposets.

Proposition~\ref{pro:convex} has a particularly appealing consequence in the case of naturally labeled posets.   Inspired by the approach of \cite{LiWe20a}, we consider posets according to their Greene shape.

\begin{definition}[\cite{Gre76}]\label{def:greene}
For a poset $P$ and $k\geq0$, let $c_k$ (resp.\ $a_k$) denote the maximum cardinality of a union of $k$ chains (resp.\ antichains) in $P$.  Then the \emph{chain Greene shape} (resp.\ \emph{antichain Greene shape}) of $P$ is the sequence $\lambda = (\lambda_1, \lambda_2,\ldots,\lambda_\ell)$ where $\lambda_k = c_k - c_{k-1}$ (resp, $\lambda_k = a_k -a_{k-1}$) and $\ell$ is the largest value of $k$ for which $\lambda_k > 0$.
\end{definition}

For example, the poset in Figure~\ref{fig:greene42} has chain Greene shape $(4,2)$ and antichain Greene shape $(2,2,1,1)$.  
\begin{figure}[htbp]
\begin{center}
\begin{tikzpicture}[scale=0.8]
\tikzstyle{every node} = [draw, shape=circle, inner sep=2pt];
\draw (0,0) node (1) {};
\draw (0,1) node (2) {};
\draw (0,2) node (4) {};
\draw (0,3) node (6) {};
\draw (-1,1) node (3) {};
\draw (1,2) node (5) {};
\draw (1) -- (2) -- (4) -- (6);
\draw (3) -- (4);
\draw (2) -- (5);
\end{tikzpicture}
\end{center}
\caption{A poset with chain Greene shape $(4,2)$ and antichain Green shape $(2,2,1,1)$.}
\label{fig:greene42}
\end{figure}
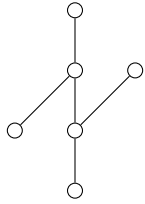

Greene proved a beautiful property of these shapes.  
\begin{theorem}[\cite{Gre76}]\label{thm:greene}
The chain Greene shape and antichain Greene shape of a poset are both partitions and are conjugates (i.e.\ transposes) of each other.
\end{theorem}

This theorem together with Proposition~\ref{pro:convex} leads to the following result.

\begin{theorem}\label{thm:antigreene}
Suppose naturally labeled posets $P$ and $Q$ satisfy $\Fsupp(P) \subseteq \Fsupp(Q)$.  Then 
\begin{enumerate}
\item the antichain Greene shape of $P$ is dominated by that of $Q$, and
\item the chain Greene shape of $P$ dominates that of $Q$.
\end{enumerate}
\end{theorem}

\begin{proof}
Applying the bar involution, our hypothesis is equivalent to having posets $\switchx{P}$ and $\switchx{Q}$ with all strict edges and which satisfy $\Fsupp(\switchx{P}) \subseteq \Fsupp(\switchx{Q})$ by Proposition~\ref{pro:fourfold}.  By~\eqref{equ:implications}, we also have $\Msupp(\switchx{P})\subseteq \Msupp(\switchx{Q})$.  If we apply Proposition~\ref{pro:convex}, observe that weak convex subposets in $\switchx{P}$ and $\switchx{Q}$ are exactly their antichains, yielding (a).  

Theorem~\ref{thm:greene} and the fact that conjugation of partitions reverses dominance order imply (b).
\end{proof}

Note that Theorem~\ref{thm:antigreene}(b) is a significant strengthening of Corollary~\ref{cor:height} in the naturally labeled case.  The former also implies that chain Greene shape equality is a necessary condition for $F$-support equality, which is \cite[Theorem~3.11]{LiWe20a}.

\section{Linear extension containment}\label{sec:hasse}

We now transition to looking at operations we can perform on a labeled poset $(P,\om)$ to yield another labeled poset $(Q,\tau)$ such that $(P,\om) \leqF (Q,\tau)$.  Referring to Theorem~\ref{thm:kexpansion} or~\eqref{equ:implications}, a condition that implies the desired inequality is $\L(P,\om) \subseteq \L(Q,\tau)$, which we refer to as ``linear extension containment.''  The goal of this section is to determine exactly when linear extension containment occurs.

One way to ensure $\L(P,\om) \subseteq \L(Q,\tau)$ is to obtain $(Q,\tau)$ from $(P,\om)$ by deleting an edge in the Hasse diagram of $(P,\om)$.  However, our computations suggest that this is a rather coarse operation, by which we mean there are then typically other labeled posets $(R,\sigma)$ such that $\L(P,\om) \subset \L(R,\sigma) \subset \L(Q,\tau)$.  Fortunately, there is an operation on Hasse diagrams with finer results that is just a little more sophisticated than deleting edges.  

It will be convenient for the remainder of this section to identify elements of labeled posets with their labels; for example, $5 <_P 2$ is expressing a relation between the elements with $\om$-labels $5$ and $2$ in some $(P,\om)$.  

\subsection{Redundancy-before-deletion}

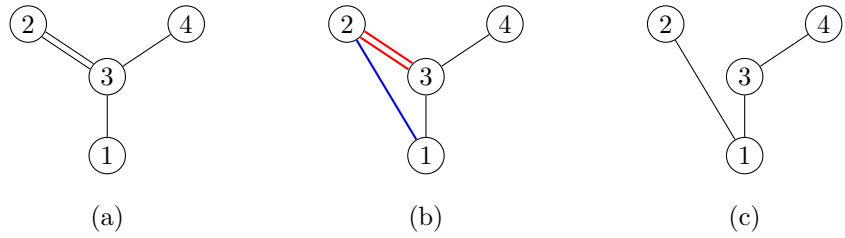
\begin{figure}
\begin{center}
\begin{tikzpicture}[scale=0.7]
\begin{scope}
\begin{scope}
\tikzstyle{every node}=[draw, shape=circle, inner sep=2pt]; 
\draw (0,0) node (a1) {1};
\draw (0,1.5) node (a3) {3};
\draw (-1.5,2.5) node (a2) {2};
\draw (1.5,2.5) node (a4) {4};
\draw (a1) -- (a3) -- (a4);
\draw[double distance=2pt] (a3) -- (a2);
\end{scope}
\draw (0,-1.2) node {(a)};
\end{scope}

\begin{scope}[xshift=40ex]
\begin{scope}
\tikzstyle{every node}=[draw, shape=circle, inner sep=2pt]; 
\draw (0,0) node (a1) {1};
\draw (0,1.5) node (a3) {3};
\draw (-1.5,2.5) node (a2) {2};
\draw (1.5,2.5) node (a4) {4};
\draw (a1) -- (a3);
\draw (a3) -- (a4);
\draw[double distance=2pt, color=red, thick] (a3) -- (a2);
\draw[thick, color=blue] (a1) -- (a2);
\end{scope}
\draw (0,-1.2) node {(b)};
\end{scope}

\begin{scope}[xshift=80ex]
\begin{scope}
\tikzstyle{every node}=[draw, shape=circle, inner sep=2pt]; 
\draw (0,0) node (a1) {1};
\draw (0,1.5) node (a3) {3};
\draw (-1.5,2.5) node (a2) {2};
\draw (1.5,2.5) node (a4) {4};
\draw (a1) -- (a3) -- (a4);
\draw (a1) -- (a2);
\end{scope}
\draw (0,-1.2) node {(c)};
\end{scope}
\end{tikzpicture}
\caption{An example of redundancy-before-deletion}
\label{fig:rbd}
\end{center}
\end{figure}

To describe this finer operation by example, consider the Hasse diagram of the labeled poset $(P,\om)$ in Figure~\ref{fig:rbd}(a). 
 Implicit in this diagram is the relation $1 <_P 2$, which we add explicitly in Figure~\ref{fig:rbd}(b).  Although Figure~4.1(b) is not a bona fide Hasse diagram since it includes this redundant edge, it is still a representation of the relations in $(P,\om)$.  We now delete the edge $3 <_P 2$ to obtain the labeled poset $(Q,\tau)$ in Figure~\ref{fig:rbd}(c).  Since the diagram in Figure~\ref{fig:rbd}(b) encodes the same set of relations among labels as $(P,\om)$, and since $(Q,\tau)$ is obtained just by deleting edges, we get  $\L(P,\om) \subseteq \L(Q,\tau)$.  

We call this operation of adding redundant edges followed by deletion of edges the \emph{redundancy-before-deletion} operation.  Let us give a general definition of this operation.

\begin{definition}\label{def:rbd}
We say that a labeled poset $(Q,\tau)$ can be obtained from a labeled poset $(P,\om)$ by \emph{redundancy-before-deletion} if the following property holds: starting with the Hasse diagram for $(P,\om)$, we can add redundant edges and then perform an ordered sequence of deletions of edges (redundant or not) to yield the Hasse diagram of $(Q,\tau)$.
\end{definition}

Implicit in our discussion is that redundant edges are designated as strict or weak as usual, according to the labels on their endpoints; we will consider a loosening of this designation in Subsection~\ref{sub:freedom}.  To see why ordering the sequence of deletions is needed in Definition~\ref{def:rbd}, see Remark~\ref{rem:deletions}.  It is clear in that if any $(Q,\tau)$ is obtained from any $(P,\om)$ by redundancy-before-deletion, then $\L(P,\om) \subseteq \L(Q,\tau)$.  Redundancy-before-deletion is a finer operation than just deletion: $(Q,\tau)$ in Figure~\ref{fig:rbd}(c) has just one more linear extension than $(P,\om)$ in Figure~\ref{fig:rbd}(a), whereas just deleting the edge $3 <_P 2$ from $(P,\om)$ would have increased the number of linear extensions by 2.  Another example of redundancy-before-deletion appears in Figure~\ref{fig:first_example}, where the two redundant edges $1 <_P 2$ and $1 <_P 4$ have been added to $(P,\om)$ before the deletion of $1 <_P 3$.

\subsection{Linear extension containment implies redundancy-before-deletion}

It is natural to ask the following question: beyond redundancy-before-deletion, what other operations can we perform on a labeled poset $(P,\om)$ to yield another labeled poset $(Q,\tau)$ such that $\L(P,\om) \subseteq \L(Q,\tau)$?  The remainder of this section is devoted to giving a full answer to this question by showing that all such $(Q,\tau)$ are obtained from $(P,\om)$ by redundancy-before-deletion; see Theorem~\ref{thm:rbd} for the precise statement.  

\!\!As a bridge between linear extension containment and redundancy-before-deletion, we will use the set of all strict relations in $(P,\om)$.  Formally, we define 
\[
\lessthan(P,\om) = \{(a,b) : a<_P b\}
\]
and call it the \emph{less-than set} of $(P,\om)$.  Recall that we are identifying elements of $(P,\om)$ with their labels, so the elements of the less-than set are ordered pairs of labels.  

\begin{theorem}\label{thm:rbd}
Let $(P,\om)$ and $(Q,\tau)$ be labeled posets.  The following are equivalent:
\begin{enumerate}
\renewcommand{\theenumi}{\arabic{enumi}}
\item $\L(P,\om) \subseteq \L(Q,\tau)$;
\item $\lessthan(P,\om) \supseteq \lessthan(Q,\tau)$;
\item $(Q,\tau)$ is obtained from $(P,\om)$ by redundancy-before-deletion.
\end{enumerate}
\end{theorem}

Before starting the proof proper of Theorem~\ref{thm:rbd}, we need some preliminary lemmas.  The first is a general property of linear extensions which is easy to check.

\begin{lemma}\label{lem:linext}
For elements $a,b$ of a labeled poset $(P,\om)$, if $a \not\leq_P b$ then there is a linear extension in which $b$ appears before $a$.  
\end{lemma}

%


The main substance of the proof of Theorem~\ref{thm:rbd} is showing that (2)$\Rightarrow$(3).  The next two lemmas detail how the edges of $\lessthan(P,\om)$ can be pared down to get $\lessthan(Q,\tau)$ in a systematic way that is consistent with edge deletion and the poset structure.

\begin{lemma}\label{lem:delete_edges}
Suppose that $(P,\om)$ is a labeled poset that includes the relation $a <_P b$.  Then $\lessthan(P,\om) \setminus \{(a,b)\}$ is a less-than set of a labeled poset if and only if $a <_P b$ is a cover relation of $(P,\om)$.
\end{lemma}

\begin{proof}
Let us add in all the reflexive relations $(x,x)$ to the set $\lessthan(P,\om) \setminus \{(a,b)\}$ and consider whether or not the resulting set is a full set of relations of a poset.  Reflexivity and antisymmetry will certainly hold, so we check transitivity, in which case we can restrict our attention to $\lessthan(P,\om) \setminus \{(a,b)\}$ without the reflexive relations.

If $(a,b)$ is not a cover relation, then there exists $c$ such that $(a,c)$ and $(c,b)$ are in $\lessthan(P,\om) \setminus \{(a,b)\}$ but $(a,b)$ is not.  Thus $\lessthan(P,\om) \setminus \{(a,b)\}$ is not transitively closed and so is not the less-than set of a labeled poset.

If $(a,b)$ is a cover relation, consider $(x,y), (y,z) \in \lessthan(P,\om) \setminus \{(a,b)\}$.  We know that $(x,z) \in \lessthan(P,\om)$ and since $(x,z)$ is not a cover relation, we get  $(x,z) \in \lessthan(P,\om) \setminus \{(a,b)\}$.  Thus $\lessthan(P,\om) \setminus \{(a,b)\}$ is the less-than set of a labeled poset.  
\end{proof}

\begin{remark}\label{rem:deletions}
Note that repeated deletions of cover relations from less-than sets as in Lemma~\ref{lem:delete_edges} can result in posets that look quite different from the original.  For example, the less-than set of the poset $(Q,\tau)$ in Figure~\ref{fig:deletions} is obtained from the less-than set of $(P,\om)$ there in the following way:
\[
\lessthan(Q, \tau) =  \lessthan(P,\om) \setminus \{(2,6)\} \setminus \{(2,5)\} \setminus \{(4,5)\},
\]
where the deletions are performed from left-to-right. Note that the last relation deleted is not a cover relation in $\lessthan(P,\om)$ but is instead a cover relation once the previous deletion has been performed.  This is an example of why ordering the sequence of edge deletions is important.
\begin{figure}
\begin{center}
\begin{tikzpicture}[scale=0.6]
\begin{scope}
\begin{scope}
\tikzstyle{every node}=[draw, shape=circle, inner sep=2pt]; 
\draw (0,0) node (a3) {3};
\draw (0,1.5) node (a4) {4};
\draw (0,3.0) node (a2) {2};
\draw (-1.0,4.5) node (a5) {5};
\draw (1.5,3) node (a1) {1};
\draw (1.0,4.5) node (a6) {6};
\draw (a5) -- (a2) -- (a6);
\draw (a3) -- (a4);
\draw[double distance=2pt] (a4) -- (a2);
\draw[double distance=2pt] (a1) -- (a6);
\end{scope}
\begin{scope}
\draw (0,-1.2) node {$(P,\om)$};
\end{scope}
\end{scope}
\begin{scope}[xshift=40ex]
\begin{scope}
\tikzstyle{every node}=[draw, shape=circle, inner sep=2pt]; 
\draw (0,0) node (a3) {3};
\draw (0,1.5) node (a4) {4};
\draw (0,3.0) node (a2) {2};
\draw (-1.5,1.5) node (a5) {5};
\draw (1.5,1.5) node (a1) {1};
\draw (1.5,3.0) node (a6) {6};
\draw (a5) -- (a3) -- (a4) -- (a6);
\draw[double distance=2pt] (a4) -- (a2);
\draw[double distance=2pt] (a1) -- (a6);
\end{scope}
\begin{scope}
\draw (0,-1.2) node {$(Q,\tau)$};
\end{scope}
\end{scope}
\end{tikzpicture}
\end{center}
\caption{$(Q,\tau)$ is obtained from $(P,\om)$ by an ordered sequence of three edge deletions. }
\label{fig:deletions}
\end{figure}
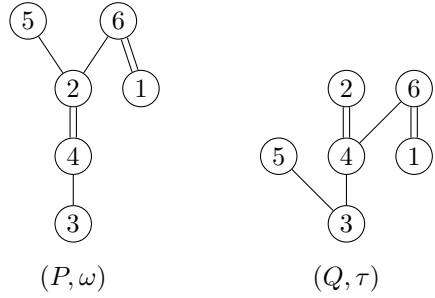
\end{remark}

Remark~\ref{rem:deletions} suggests that we need to be mindful when talking about deletions of edges from less-than sets, hence the need for the next lemma.

\begin{lemma}\label{lem:delete_covers}
For labeled posets $(P,\om)$ and $(Q,\tau)$, we have $\lessthan(P,\om) \supseteq \lessthan(Q,\tau)$ if and only if $\lessthan(Q,\tau)$ can be obtained from $\lessthan(P,\om)$ by an ordered sequence of deletions of cover relations.
\end{lemma}

\begin{proof}
Since the ``if'' direction is clear, we suppose $\lessthan(P,\om) \supset \lessthan(Q,\tau)$.  Because of this strict inequality and since cover relations generate the entire less-than set, there must exist a cover relation $a <_P b$ of $(P,\om)$ that is not an element of $\lessthan(Q,\tau)$.  Deleting this cover relation as in Lemma~\ref{lem:delete_edges} will yield a labeled poset $(P',\om')$ that has one less order relation than $(P,\om)$, and so $\lessthan(P',\om') \supseteq \lessthan(Q,\tau)$.  Repeated applications of this deletion process will result in $\lessthan(Q,\tau)$.
\end{proof}

\begin{proof}[Proof of Theorem~\ref{thm:rbd}]
We will prove (1)$\Rightarrow$(2)$\Rightarrow$(3)$\Rightarrow$(1).

To show (1)$\Rightarrow$(2), assume that $\L(P,\om) \subseteq \L(Q,\tau)$ and suppose that there exists $(a,b) \in \lessthan(Q,\tau)$ such that $(a,b) \not\in \lessthan(P,\om)$.  By Lemma~\ref{lem:linext}, there exists a linear extension $\pi \in \L(P,\om)$ in which $b$ appears before $a$.  But then $\pi \in  \L(Q,\tau)$, which implies $(a,b) \not\in \lessthan(Q,\tau)$, a contradiction.

To show (2)$\Rightarrow$(3), start with the Hasse diagram for $(P,\om)$ and add in every possible redundant edge; the resulting set of edges will be  $\lessthan(P,\om)$.  Since $\lessthan(P,\om) \supseteq \lessthan(Q,\tau)$, we can apply Lemma~\ref{lem:delete_covers}: $\lessthan(Q,\tau)$ can be obtained from $\lessthan(P,\om)$ using an ordered sequence of edge deletions.  Then simply delete from $\lessthan(Q,\tau)$ all redundant edges in any order to obtain the Hasse diagram for $(Q,\tau)$.  In summary, the Hasse diagram $(Q,\tau)$ is obtained from $(P,\om)$ by adding redundant edges followed by an ordered sequence of deletions of edges, as required.

We have already mentioned the reasons for (3)$\Rightarrow$(1): adding redundant edges to the Hasse diagram of $(P,\om)$ doesn't change the set of linear extensions, nor does deleting redundant edges, while deleting cover relations adds elements to the set of linear extensions.   Thus $\L(P,\om) \subseteq \L(Q,\tau)$.
\end{proof}

\subsection{Adding redundant edges more freely}\label{sub:freedom}

In redundancy-before-deletion, the strictness or weakness of the redundant edges added to $(P,\om)$ is determined by the labeling $\om$.  For example, the redundant edge added in Figure~\ref{fig:rbd}(b) is weak as dictated by the labels 1 and 2.  However, particularly in Section~\ref{sec:greene}, we want the freedom to deal with labeled posets where instead of a labeling $\om$ being displayed, we only have some assignment of strict and weak edges consistent with an underlying $\om$.  In this situation, are there any restrictions on the strictness and weakness of the redundant edges we add?  For example, are the strictness and weakness of the redundant edges in the examples in Figure~\ref{fig:valid} valid, in the sense that they are consistent with an underlying labeling? 
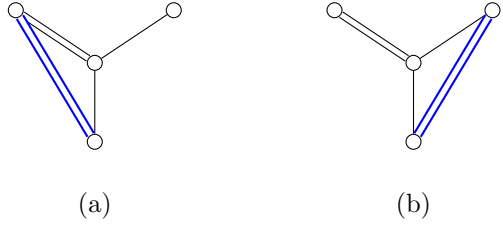
\begin{figure}
\begin{center}
\begin{tikzpicture}[scale=0.7]
\begin{scope}
\begin{scope}
\tikzstyle{every node}=[draw, shape=circle, inner sep=2pt]; 
\draw (0,0) node (a1) {};
\draw (0,1.5) node (a3) {};
\draw (-1.5,2.5) node (a2) {};
\draw (1.5,2.5) node (a4) {};
\draw (a1) -- (a3) -- (a4);
\draw[double distance=2pt] (a3) -- (a2);
\draw[double distance=2pt, thick, color=blue] (a1) -- (a2);
\end{scope}
\begin{scope}
\draw (0,-1.2) node {(a)};
\end{scope}
\end{scope}
\begin{scope}[xshift=40ex]
\begin{scope}
\tikzstyle{every node}=[draw, shape=circle, inner sep=2pt]; 
\draw (0,0) node (a1) {};
\draw (0,1.5) node (a3) {};
\draw (-1.5,2.5) node (a2) {};
\draw (1.5,2.5) node (a4) {};
\draw (a1) -- (a3) -- (a4);
\draw[double distance=2pt] (a3) -- (a2);
\draw[double distance=2pt, thick, color=blue]  (a1) -- (a4);
\end{scope}
\begin{scope}
\draw (0,-1.2) node {(b)};
\end{scope}
\end{scope}
\end{tikzpicture}
\caption{Do these assignments of strict and weak edges come from an underlying labeling?  (The answers are ``yes'' for (a) but ``no'' for (b).)}
\label{fig:valid}
\end{center}
\end{figure}

To state the restriction on the strictness and weakness of redundant edges, we need to define a \emph{bad cycle}.  Given a Hasse diagram of a poset along with some redundant edges, and an assignment of strictness and weakness to all the edges, build a directed graph by orienting all weak edges upwards and all strict edges downwards.  If the resulting digraph has a directed cycle, then we say that the assignment of strict and weak edges has a \emph{bad cycle}.  The reader is invited to construct these digraphs for (a) and (b) in Figure~\ref{fig:valid} and to refer to these digraphs in the proof of Proposition~\ref{pro:badcycles}

\begin{proposition}\label{pro:badcycles}
Given the Hasse diagram of a poset along with some redundant edges, an assignment of strict and weak edges is consistent with an underlying labeling $\om$ if and only if the assignment has no bad cycles.
\end{proposition}

\begin{proof}
Let $D$ denote the digraph associated with an assignment of strict and weak edges to a poset $P$.

Suppose first that $D$ has a directed cycle $a_1 \to a_2 \to \cdots \to a_k \to a_1$, and that the assignment of strict and weak edges comes from a labeling $\om$.  Notice that the digraph is defined so that its edges always point to the higher $\om$-label.  Thus $\om(a_1) < \om(a_2) < \ldots < \om(a_k) < \om(a_1)$, a contradiction.

For the converse, we know $D$ is acyclic, and we will explicitly define a labeling $\om$ of the poset $P$.  Define a new partial order $R$ on the elements of $P$ by saying $a \leq_R b$ if there is a directed path in $D$ from $a$ to $b$.  We see that $R$ is a partial order, so it can be given a natural labeling $\om$.  These $\om$-labels of the elements of $P$ respect the assignment of strict and weak edges.  Indeed, if the edge $a <_P b$ is strict, then $b <_R a$, so $\om(b) < \om(a)$, and similarly for weak edges.
\end{proof}

So given a labeled poset $(P,\om)$ we need not just add redundant edges whose weakness/strictness is determined by the labeling $\om$, but we can more generally add any redundant edges whose weakness/strictness does not create a bad cycle.  We call the addition of the latter type of edges the \emph{generalized-redundancy} operation, and use the term \emph{generalized-redundancy-before-deletion} for the resulting generalization of redundancy-before-deletion.

We would like to develop an analogue of Theorem~\ref{thm:rbd} for generalized redundancy.  As we will see, the crux of the matter is that generalized-redundancy-before-deletion applied to $(P,\om)$ is equivalent to regular redundancy-before-deletion applied to an appropriately relabeled $(P,\om)$.  A relabeling $\om'$ of a labeled poset $(P,\om)$ is said to be \emph{consistent} if the Hasse diagram of $(P,\om')$ has the same set of weak (resp.\ strict) edges as the Hasse diagram of $(P,\om)$. Importantly, $K_{(P,\om)} = K_{(P,\om')}$ since, as we observed, Definition~\ref{def:popartition} is unchanged if we replace $a <_P b$ with $a \prec_P b$.  Consequently, Theorem~\ref{thm:rbd2} can be applied to questions of $F$-positivity, as we do repeatedly in Section~\ref{sec:greene}.

\begin{theorem}\label{thm:rbd2}
Let $(P,\om)$ and $(Q,\tau)$ be labeled posets.  The following are equivalent:
\begin{enumerate}
\renewcommand{\theenumi}{\arabic{enumi}}
\item $\L(P,\om') \subseteq \L(Q,\tau)$ for some consistent relabeling $\om'$ of $(P, \om)$;
\item $\lessthan(P,\om') \supseteq \lessthan(Q,\tau)$ for some consistent relabeling $\om'$ of $(P, \om)$;
\item $(Q,\tau)$ is obtained from $(P,\om)$ by generalized-redundancy-before-deletion.
\end{enumerate}
\end{theorem}

\begin{proof}
We will show that (3) is equivalent to the following statement:
\begin{enumerate}
\item[$(3')$]$(Q,\tau)$ is obtained from $(P,\om')$ by (regular) redundancy-before-deletion for some consistent relabeling $\om'$ of $(P, \om)$. 
\end{enumerate}
Then the result will follow by applying Theorem~\ref{thm:rbd} with $\om'$ in place of $\om$.  

We need to show that both the process of $(3)$ and that of ($3'$) can arrive at the same possibilities for $(Q,\tau)$ when starting at a given $(P,\om)$.  Temporarily ignoring labels and weakness/strictness, we note that whether we apply the process of (3) or ($3'$), we begin by adding some set of redundant edges to $P$.  And (3) and ($3'$) can certainly mimic each other in the edges they add to $P$.  We have two ways to designate the weakness/strictness of the redundant edges: following (3), we can designate in any way that does not create a bad cycle, or, following ($3'$), we can designate according to a labeling $\om'$. 

If we start with $(P,\om)$ and follow (3) in adding redundant edges, by Proposition~\ref{pro:badcycles}, the same designation of weakness/strictness can be achieved by following ($3'$) using an appropriate relabeling $\om'$.  We know that $\om'$ must be a consistent labeling of $(P,\om)$ because adding edges as in (3) does not change the weakness/strictness of the cover relations of $(P,\om)$. 

If we instead follow ($3'$) in adding edges and designate weakness/strictness according to the labeling $\om'$, we will not introduce any bad cycles, again by Proposition~\ref{pro:badcycles}.  Therefore, any designation obtained by following ($3'$) can be mimicked by following the generalized redundancy of (3).

In summary, after adding redundant edges, both (3) and ($3'$) can arrive at the same collection of weak (resp.\ strict) redundant edges and cover relations of $(P,\om)$.  Since deletion for (3) and ($3'$) are performed in exactly the same way, they can certainly continue to  mimic each other and arrive at the same $(Q,\tau)$.  
\end{proof}

\section{Combining posets}\label{sec:ur}

In Section 4, we started with a single labeled poset $(P,\om)$ and performed operations to obtain another labeled poset $(Q,\tau)$.  In this section, we want to consider the situation where we start with a pair (or multiple pairs) of posets $(P,\om)$ and $(Q,\tau)$ such that $(P,\om) \leqF (Q,\tau)$ and ask what operations can we perform on both posets to preserve the $F$-positivity.   

Operations that combine posets $P$ and $Q$ include the disjoint union $P+Q$, the ordinal sum $P \oplus Q$, and the ordinal product $P \otimes Q$; for definitions, see \cite[Sec.~3.2]{ec1e2}.  We consider an operation that is a common generalization of all three of these operations, defined in \cite{BHK18+,BHK18} as the ``Ur-operation.''  We will refer to the operation as ``poset assembly.''

\subsection{Poset assembly}

At this point, it will make things clearer if we slightly loosen our convention of identifying poset elements with their labels; specifically, the letters $p$ and $q$ used below refer to specific elements of posets rather than positive integers. 

\begin{definition}[\cite{BHK18+,BHK18}]\label{def:ur}
For a labeled poset $(\P,\om)$ and a sequence of posets $(P_1, \ldots, P_{\vert \P}\vert )$ on disjoint sets, we define the \emph{assembled poset} $\urp$ to be the disjoint union $\bigcup_{i=1}^{\vert \P\vert } P_i$ with the following order relation:
\[
\mbox{for $p \in P_j$ and $q\in P_k$, we have $p \leq q$ when} \left\{ \begin{array}{ll} 
j <_{\P} k  & \mbox{\ if\ } j\neq k \\
p \leq_{P_j} q & \mbox{\ if\ } j=k 
\end{array} 
\right..
\]
\end{definition}

An example of $\urp$ is shown in Figure~\ref{fig:ur}, with the labeling explained next.  We refer to $\P$ as the \emph{framework poset} and call the $P_i$ the \emph{component posets}.  Roughly speaking, the assembled poset is obtained by replacing the element of $\P$ labeled $i$  with $P_i$ for all $i$.

\begin{figure}[htbp]
\begin{center}
\begin{tikzpicture}
\begin{scope}[scale=0.8]
\begin{scope}[xshift=-30 ex]
\begin{scope}[draw, shape=circle, inner sep=2pt]
\node at (-1,0) [draw,circle] (P1) {$1$};
\node at (0,1) [draw,circle] (P2) {$2$};
\node at (1,0) [draw,circle] (P3) {$3$};
\draw[double distance=2pt] (P3)--(P2);
\end{scope}
\node at (0,-0.8) {$(\P,\om)$};
\draw (P1)--(P2);
\end{scope}
\begin{scope}
\begin{scope}[draw, shape=circle, inner sep=2pt]
\node at (0,2) [draw,circle] (P2) {$2$}; 
\node at (-0.7,1) [draw,circle] (P1) {$1$};
\node at (0.7,1) [draw,circle] (P3) {$3$};
\node at (0,0) [draw,circle] (P4) {$4$};
\draw[double distance=2pt] (P3)--(P2);
\end{scope}
\node at (0,-0.8) {$(P_1,\om_1)$};
\draw (P1)--(P2);
\draw[double distance=2pt] (P4)--(P3);
\draw[double distance=2pt] (P4)--(P1);
\end{scope}
\begin{scope}[xshift=20ex]
\begin{scope}[draw, shape=circle, inner sep=2pt]
\node at (0,0) [draw,circle] (P3) {$3$};
\node at (-0.7, 1) [draw,circle] (P1) {$1$};
\node at (0.7,1) [draw,circle] (P4) {$4$};
\node at (-0.7,2) [draw,circle] (P2) {$2$};
\draw (P3)--(P4);
\draw[double distance=2pt] (P3)--(P1);
\draw (P1)--(P2);
\end{scope}
\node at (0,-0.8) {$(P_2,\om_2)$};
\end{scope}
\begin{scope}[xshift=35ex]
\begin{scope}[draw, shape=circle, inner sep=2pt]
\node at (0,0) [draw,circle] (P1) {$1$};
\node at (1,0) [draw,circle] (P3) {$3$};
\node at (0,1) [draw,circle] (P2) {$2$};
\node at (1,1) [draw,circle] (P4) {$4$};
\draw (P1)--(P4);
\draw (P1)--(P2);
\draw (P3)--(P4);
\draw[double distance=2pt] (P3)--(P2);
\end{scope}
\node at (0.5,-0.8) {$(P_3,\om_3)$};
\end{scope}
\end{scope}

\begin{scope}[scale=1.0,yshift=-45ex, shape=circle, inner sep=1pt]
\node at (0,3) [draw] (P23) {2,3};
\node at (-0.7,4) [draw] (P21) {2,1};
\node at (0.7,4) [draw] (P24) {2,4};
\node at (-0.7,5) [draw] (P22) {2,2};
\draw (P23)--(P24);
\draw[double distance=2pt] (P23)--(P21);
\draw (P21)--(P22);

\node at (0-1,2) [draw] (P12) {1,2}; 
\node at (-0.7-1,1) [draw] (P11) {1,1};
\node at (0.7-1,1) [draw] (P13) {1,3};
\node at (0-1,0) [draw] (P14) {1,4};
\draw[double distance=2pt] (P13)--(P12);
\draw (P11)--(P12);
\draw[double distance=2pt] (P14)--(P13);
\draw[double distance=2pt] (P14)--(P11);

\node at (0+1.3,0) [draw] (P31) {3,1};
\node at (1+1.3,0) [draw] (P33) {3,3};
\node at (0+1.3,1) [draw] (P32) {3,2};
\node at (1+1.3,1) [draw] (P34) {3,4};
\draw (P31)--(P34);
\draw (P31)--(P32);
\draw (P33)--(P34);
\draw[double distance=2pt] (P33)--(P32);

\draw[double distance=2pt] (P32)--(P23);
\draw[double distance=2pt] (P34)--(P23);
\draw (P12)--(P23);
\draw (P12)--(P23);

\end{scope}
\end{tikzpicture}
\caption{A framework labeled poset $(\P,\om)$, three component labeled posets, and the resulting assembled poset $\urpom$ with the inherited labeling.}
\label{fig:ur}
\end{center}
\end{figure}
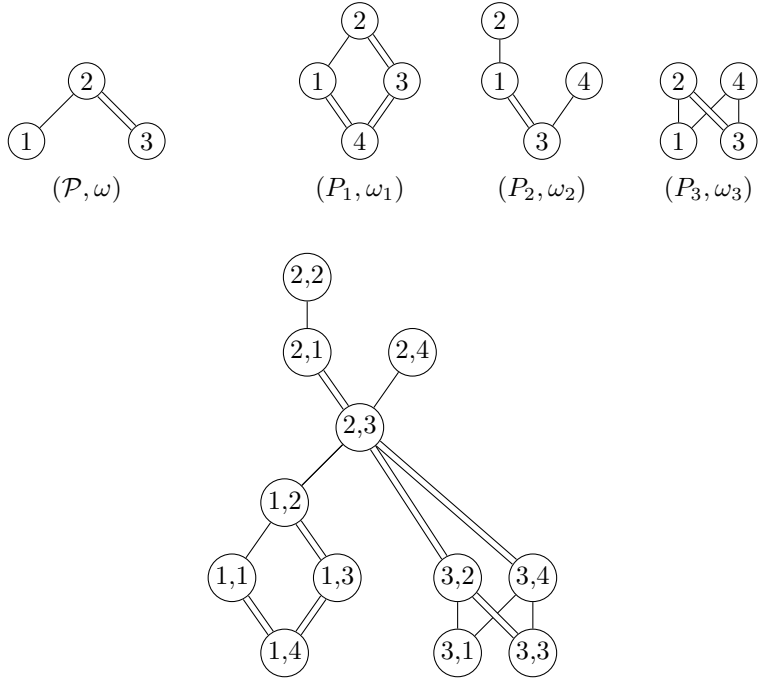

\begin{definition}\label{def:inherited}
For a framework labeled poset $(\P,\om)$ and a sequence of component labeled posets $((P_1,\om_1),\ldots,(P_{\vert \P\vert },\om_{\vert \P\vert }))$, we define the \emph{inherited labeling} $\Om$ of $\urp$ as follows: 
\[
\text{for } p \in P_k,~\Om(p) = ( k,  \om_k(p)).
\]
\end{definition}

See Figure~\ref{fig:ur} for an example, where we drop the parentheses around the inherited labels to allow for a more compact figure. We totally order the inherited labels according to lexicographic order, allowing us to refer to weak and strict edges in the assembled poset $\urpom$.  Note that the inherited labeling is defined so that it preserves the weak and strict edges within each component poset, and that the weakness/strictness between the components matches that in $(P,\om)$.  Since our assembled posets $\urp$ will always be labeled by the inherited labeling, we will abbreviate $\urpom$ as $\urp$.

Theorem~\ref{thm:kexpansion} can be used to compute $\kurp$: we can construct the linear extensions $\pi$ of $\urp$ as sequences of pairs of positive integers, and define the descents of $\pi$ using lexicographic order.  For example, one linear extension of 
$\urp$ from Figure~\ref{fig:ur} is 
\begin{equation}\label{equ:ur_lin_ext}
 \pi = ((1,4), (1,1), (3,1), (1,3), (3, 3), (1,2), (3,2), (3,4), (2,3), (2,4), (2,1), (2,2))
\end{equation}
which has descent composition $122322$.  As would be expected, we will use $\L(\urp)$ to denote the set of linear extensions of an assembled poset $\urp$.  

\subsection{Poset assembly preserves linear extension containment}

Following the approach of Section~\ref{sec:hasse}, Theorem~\ref{thm:kexpansion} tells us that two assembled posets $\urp$ and $\urq$ with inherited labelings will satisfy $\urp \leqF \urq$ if $\L(\urp) \subseteq \L(\urq)$.  Our first result about poset assembly is that linear extension containment is preserved in full generality.  

The proof will make use of the less-than set of $\urp$ which, by Definitions~\ref{def:ur} and~\ref{def:inherited}, is given by: 
$((j_1, k_i) , (j_2, k_2)) \in \lessthan(\urp)$ if and only if $(j_1, j_2) \in \lessthan(\P)$ or both $j_1 =j_2$ and $(k_1, k_2) \in \lessthan(P_{j_1})$.

\begin{theorem}\label{thm:ur_lin_ext}
Consider labeled posets $(\P,\om)$ and $(\CQ,\tau)$ such that $\L(\P,\om) \subseteq \L(\CQ,\tau)$.   If sequences $((P_1,\om_1),\ldots,(P_{\vert \P}\vert ,\om_{\vert \P}\vert ))$ and $((Q_1,\tau_1),\ldots,(Q_{\vert \P}\vert ,\tau_{\vert \P}\vert ))$ of labeled posets satisfy $\L(P_r,\om_r) \subseteq \L(Q_r,\tau_r)$ for all $r$, then 
\[
\L(\urp) \subseteq \L(\urq).
\]
\end{theorem}

\begin{proof}
By Theorem~\ref{thm:rbd}, we know $\lessthan(\P,\om) \supseteq \lessthan(\CQ,\tau)$ and $\lessthan(P_r,\om_r) \supseteq \lessthan(Q_r,\tau_r)$ for all $r$, and we wish to show that $\lessthan(\urp) \supseteq \lessthan(\urq)$.

Let $((j_1, k_i) , (j_2, k_2)) \in \lessthan(\urq)$.  There are two cases to consider. \linebreak If $j_1 \neq j_2$, then $(j_1, j_2) \in \lessthan(\CQ,\tau)$, and hence $(j_1, j_2) \in \lessthan(\P,\om)$, implying \linebreak $((j_1, k_i) , (j_2, k_2)) \in \lessthan(\urp)$.  If $j_1 = j_2$, then $(k_1, k_2) \in \lessthan(Q_{j_1},\tau_{j_1})$, and hence $(k_1, k_2) \in \lessthan(P_{j_1},\om_{j_1})$, implying $((j_1, k_i) , (j_2, k_2)) \in \lessthan(\urp)$.
\end{proof}

\subsection{Poset assembly and \texorpdfstring{$F$}{F}-positivity}\label{sub:ur_positivity}

Having shown that poset assembly preserves linear extension containment, we next ask whether it preserves $F$-positivity.  This is not always the case; for example, see Figure~\ref{fig:ur_counterexample}.  We have $(\P,\om) \leqF (\CQ,\tau)$ but it is not that case that $\P[i \to P_i] \leqF \CQ[i \to P_i]$.
\begin{figure}[htbp]
\begin{center}
\begin{tikzpicture}
\begin{scope}[xshift=0ex]
\begin{scope}[draw, shape=circle, inner sep=2pt]
\node at (0,0) [draw,circle] (P3) {$2$};
\node at (-0.7, 1) [draw,circle] (P2) {$4$};
\node at (0.7,1) [draw,circle] (P4) {$1$};
\node at (0,2) [draw,circle] (P1) {$3$};
\draw (P4)--(P1);
\draw (P2) -- (P3);
\draw[double distance=2pt] (P1)--(P2);
\end{scope}
\node at (0,-0.8) {$(\P,\om)$};
\end{scope}
\begin{scope}[xshift=15ex]
\begin{scope}[draw, shape=circle, inner sep=2pt]
\node at (0,0) [draw,circle] (P2) {$3$};
\node at (1,0) [draw,circle] (P4) {$1$};
\node at (0,1) [draw,circle] (P1) {$2$};
\node at (1,1) [draw,circle] (P3) {$4$};
\draw (P2)--(P3) -- (P4) -- (P1);
\draw[double distance=2pt] (P2)--(P1);
\end{scope}
\node at (0.5,-0.8) {$(\CQ, \tau)$};
\end{scope}
\begin{scope}[xshift=40ex]
\begin{scope}[draw, shape=circle, inner sep=2pt]
\node at (0,0) [draw,circle] (P1) {$3$};
\node at (1,0) [draw,circle] (P2) {$1$};
\node at (1,1) [draw,circle] (P3) {$2$};
\draw (P2)--(P3);
\end{scope}
\node at (0.5,-0.8) {$(P_i, \om_i)$ for all $i$};
\end{scope}
\end{tikzpicture}
\end{center}
\caption{We have $(\P,\om) \leqF (\CQ,\tau)$ but $\P[i \to P_i] \not\leqF \CQ[i \to P_i]$. }
\label{fig:ur_counterexample}
\end{figure}
In this example it is even the case that $\Fsupp\left(\P[i \to P_i]\right) \not\subseteq \Fsupp\left(\CQ[i\to P_i]\right)$: the poset $\P[i\to P_i]$ has a linear extension 
\[ \left((1,1),(2,1),(2,2),(2,3),(4,1),(4,3),(4,2),(1,3),(1,2),(3,1),(3,2),(3,3)\right),\]
contributing $6114$ to the $F$-support of $\P[i\to P_i]$. As for $\CQ[i\to P_i]$, it is not difficult to check that a linear extension beginning with five ascents cannot be followed by three consecutive descents, so $6114$ is not in the $F$-support of $\CQ[i\to P_i]$.

We can achieve $F$-positivity preservation if we strengthen our hypotheses to a mixture of linear extension containment and $F$-positivity.  This brings us to our second result on poset assembly, whose proof is considerably more technical than that of Theorem~\ref{thm:ur_lin_ext}.

\begin{theorem}\label{thm:ur_fpositivity}
Consider labeled posets $(\P,\om)$ and $(\CQ,\tau)$ such that $\L(\P,\om) \subseteq \L(\CQ,\tau)$.
If sequences $((P_1,\om_1),\ldots,(P_{\vert \P\vert },\om_{\vert \P\vert }))$ and $((Q_1,\tau_1),\ldots,(Q_{\vert \P\vert },\tau_{\vert \P\vert }))$ of labeled posets satisfy $(P_r,\om_r) \leqF (Q_r,\tau_r)$ for all $r$, then 
\[
\urp \leqF \urq.
\]
\end{theorem}

As special cases, we can let $\P =\CQ$ be a 2-element antichain to deduce that disjoint union preserves $F$-positivity, or let $\P=\CQ$ be a 2-element chain (with either a strict or a weak edge) to deduce that ordinal sum does too.

\begin{proof}[Proof of Theorem~\ref{thm:ur_fpositivity}]  In view of Theorem~\ref{thm:kexpansion}, we will prove the inequality by defining an injective map $\Phi:\L(\urp) \to \L(\urq)$ that preserves the descent set (where descents are defined according to lexicographic order).  Let 
\[
\pi = ((j_1,k_1), (j_2,k_2), \ldots, (j_N, k_N))
\]
be a linear extension of $\urp$.  For $r \in [\vert \P\vert ]$, let $\pi^{(r)}$  denote the subsequence of $\pi$ consisting of those pairs $(j,k)$ such that $j = r$:
\[
\pi^{(r)}= ((r, k_{\ell_1}), (r, k_{\ell_2}), \ldots, (r, k_{\ell_{\vert P_r\vert }})).
\]
For example, for $\pi$ from~\eqref{equ:ur_lin_ext}, we get $\pi^{(3)} = ((3,1), (3,3), (3,2), (3,4))$.  Note that $(k_{\ell_1}, k_{\ell_2}, \ldots,k_{\ell_{\vert P_r\vert }})$ is a linear extension of $P_r$\,.  By our hypothesis, there exists a descent-preserving injective map $\varphi^{(r)}$ from $\L(P_r)$ to $\L(Q_r)$.  Denote the image of $(k_{\ell_1}, k_{\ell_2}, \ldots,k_{\ell_{\vert P_r\vert }})$ under $\varphi^{(r)}$ by $(k'_{\ell_1}, k'_{\ell_2}, \ldots,k'_{\ell_{\vert P_r\vert }})$.  Then let $\phi^{(r)}$ be the map that sends $\pi^{(r)}$ to 
\[
\phi^{(r)}(\pi^{(r)})= ((r, k'_{\ell_1}), (r, k'_{\ell_2}), \ldots, (r, k'_{\ell_{\vert P_r\vert }}));
\]
see Example~\ref{exa:ur_fpositivity} below.  Finally, since every entry $(j,k)$ of $\pi$ is a member of exactly one $\pi^{(r)}$, we define $\Phi(\pi)$ to be the result of applying every $\phi^{(r)}$ to the subsequence $\pi^{(r)}$ within $\pi$.
 We need to show that $\Phi$ is well defined, injective and descent-preserving.

For $\pi \in \L(\urp)$, we first need to show that $\Phi(\pi) \in \L(\urq)$.  Clearly the pairs in $\Phi(\pi)$ are in bijection with the labels on $\urq$,  So consider $(a_1, b_1)$ that comes before $(a_2, b_2)$ in $\Phi(\pi)$.  We need to show that $((a_2,b_2),(a_1,b_1)) \not\in \lessthan(\urq)$. 

Suppose first that $a_1 \neq a_2$.  Since the $a_j$ labels do not change under $\Phi$, there must be some $(a_1, b_3)$ that comes before some $(a_2, b_4)$ in $\pi$.  Thus $(a_2, a_1) \not\in \lessthan(\P)$ so, by Theorem~\ref{thm:rbd} and since $\L(\P,\om) \subseteq \L(\CQ,\tau)$, we know $(a_2, a_1) \not\in \lessthan(\CQ)$.  Thus $((a_2,b_2),(a_1,b_1)) \not\in \lessthan(\urq)$.  

Now suppose that $a_1 = a_2$.  Since $(a_1, b_1)$ comes before $(a_2, b_2)$ in $\Phi(\pi)$, by construction of $\Phi(\pi)$, we know that $b_1$ comes before $b_2$ in some linear extension of $(Q_{a_1}, \om_{a_1})$.  Thus $(b_2, b_1) \not\in \lessthan{(Q_{a_1}, \om_{a_1})}$ and so $((a_2,b_2),(a_1,b_1)) \not\in \lessthan(\urq)$.

To show injectivity, suppose $\Phi(\pi) = \Phi(\rho)$ for $\rho \in \L(\urp)$.  Since $\Phi(\pi)$ does not change the first entry in each pair of $\pi$, we know the $j$th pair of $\pi$ matches the $j$th pair $\rho$ in their first entries, for all $j$.  The second entries also match since each $\varphi^{(r)}$ in injective, and so each $\phi^{(r)}$ is too.  

To show the descent set is preserved, consider adjacent pairs $(j_\ell, k_\ell)$ and \linebreak $(j_{\ell+1}, k_{\ell+1})$ in $\pi = ((j_1,k_1), (j_2,k_2), \ldots, (j_N, k_N))$.  If $j_\ell \neq j_{\ell+1}$, then since $\Phi$ preserves the first entry of each pair, $\ell$ is a descent of $\pi$ if and only if it is a descent of $\Phi(\pi)$.  The same conclusion applies if $j_\ell = j_{\ell+1}$ since each $\varphi^{(r)}$ in descent preserving, and so each $\phi^{(r)}$ is too.  
\end{proof}

\begin{example}\label{exa:ur_fpositivity}
Let $\urp$ as be as in Figure~\ref{fig:ur}.  Let $(Q,\tau)$ be the result of deleting the edge $1<2$ from $(P,\om)$.  For $r=1,2,3$, define $(Q_r,\tau_r) = (P_r,\om_r)$ except delete the edge $1<4$ from $(Q_3, \tau_3)$.  As in~\eqref{equ:ur_lin_ext}, let
\[
\pi =(\textcolor{red}{(1,4)}, \textcolor{red}{(1,1)}, \textcolor{blue}{(3,1)}, \textcolor{red}{(1,3)}, \textcolor{blue}{(3, 3)}, \textcolor{red}{(1,2)}, \textcolor{blue}{(3,2)}, \textcolor{blue}{(3,4)}, (2,3), (2,4), (2,1), (2,2)).
\]
Then let us choose the descent-preserving $\varphi^{(r)}$ so that
\begin{align*}
\phi^{(1)}(\pi^{(1)}) &= \phi^{(1)}(\textcolor{red}{((1,4),(1,1),(1,3),(1,2))}) = \textcolor{red}{((1,4),(1,1),(1,3),(1,2))},\\
\phi^{(2)}(\pi^{(2)}) &= \phi^{(2)}(((2,3),(2,4),(2,1),(2,2))) = ((2,3),(2,4),(2,1),(2,2)), \\
\phi^{(3)}(\pi^{(3)}) &= \phi^{(3)}(\textcolor{blue}{((3,1),(3,3),(3,2),(3,4))}) = \textcolor{blue}{((3,3),(3,4),(3,1),(3,2))}
\end{align*}
(in fact, $\varphi^{(3)}$ is the only one in which we had some choice).  Then 
\[
\Phi(\pi) = (\textcolor{red}{(1,4)}, \textcolor{red}{(1,1)}, \textcolor{blue}{(3,3)}, \textcolor{red}{(1,3)}, \textcolor{blue}{(3, 4)}, \textcolor{red}{(1,2)}, \textcolor{blue}{(3,1)}, \textcolor{blue}{(3,2)}, (2,3), (2,4), (2,1), (2,2)).
\]
We see that $\des(\pi) = \des(\Phi(\pi))$ and can check that $\Phi(\pi) \in \L(\urq)$.
\end{example}

\section{Classes with conditions that are both necessary and sufficient}\label{sec:greene}

This section is devoted to classes of posets for which we have simple conditions on $(P,\om)$ and $(Q,\tau)$ that are both necessary and sufficient for $(P,\om) \leqF (Q,\tau)$. 

\subsection{Greene shape \texorpdfstring{$(k,1)$}{(k,1)}}\label{subsec:Greene}

For this subsection, we will restrict our attention to naturally labeled posets and so we will denote $(P,\om)$ just by $P$.  We will use the term ``Greene shape'' hereafter to mean chain Greene shape (from Definition~\ref{def:greene}). Posets of Greene shape $(k,1)$ are those that consist of a maximal chain with $k$ elements, which we call the \emph{spine}, and a single other element $e$, which we call the \emph{foot}.  Such posets come in four types, with examples from the four types when $k=5$ illustrated in Figure~\ref{fig:4types}:\footnote{In \cite{LeMc22}, these four types are labeled I--IV, respectively.}
\begin{enumerate}
\item[(D)] $e$ is covered by a non-minimal element of the spine (the foot points \textbf{D}own);
\item[(U)] $e$ covers a non-maximal element of the spine (the foot points \textbf{U}p);
\item[(B)] $e$ both covers and is covered by elements of the spine (\textbf{B}oth up and down);
\item[(I)] $e$ is its own connected component (\textbf{I}solated).
\end{enumerate}

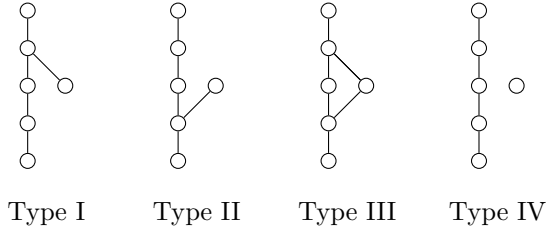
\begin{figure}[ht]
\begin{center}
\begin{tikzpicture}
\begin{scope}
\tikzstyle{every node} = [draw, shape=circle, inner sep=2pt];
\node at (0,0) [draw,circle]   (P1) {};
\node at (0,0.5) [draw,circle] (P2) {};
\node at (0,1) [draw,circle]   (P3) {};
\node at (0,1.5) [draw,circle] (P4) {};
\node at (0,2) [draw,circle]   (P5) {};
\node at (0.5,1) [draw,circle]   (P6) {};
\draw (P6)--(P4);
\draw (P1)--(P2);
\draw (P2)--(P3);
\draw (P3)--(P4);
\draw (P4)--(P5);
\end{scope}
\begin{scope}
\node at (0.25,-0.7) {Type D};
\end{scope}

\begin{scope}
\tikzstyle{every node} = [draw, shape=circle, inner sep=2pt];
\node at (2,0) [draw,circle]   (Q1) {};
\node at (2,0.5) [draw,circle] (Q2) {};
\node at (2,1) [draw,circle]   (Q3) {};
\node at (2,1.5) [draw,circle] (Q4) {};
\node at (2,2) [draw,circle]   (Q5) {};
\node at (2.5,1) [draw,circle]   (Q6) {};
\draw (Q6)--(Q2);
\draw (Q1)--(Q2);
\draw (Q2)--(Q3);
\draw (Q3)--(Q4);
\draw (Q4)--(Q5);
\end{scope}
\begin{scope}
\node at (2.25,-0.7) {Type U};
\end{scope}

\begin{scope}
\tikzstyle{every node} = [draw, shape=circle, inner sep=2pt];
\node at (4,0) [draw,circle]   (R1) {};
\node at (4,0.5) [draw,circle] (R2) {};
\node at (4,1) [draw,circle]   (R3) {};
\node at (4,1.5) [draw,circle] (R4) {};
\node at (4,2) [draw,circle]   (R5) {};
\node at (4.5,1) [draw,circle]   (R6) {};
\draw (R6)--(R4);
\draw (R6)--(R2);
\draw (R6)--(R4);
\draw (R1)--(R2);
\draw (R2)--(R3);
\draw (R3)--(R4);
\draw (R4)--(R5);
\end{scope}
\begin{scope}
\node at (4.25,-0.7) {Type B};
\end{scope}

\begin{scope}
\tikzstyle{every node} = [draw, shape=circle, inner sep=2pt];
\node at (6,0) [draw,circle]   (S1) {};
\node at (6,0.5) [draw,circle] (S2) {};
\node at (6,1) [draw,circle]   (S3) {};
\node at (6,1.5) [draw,circle] (S4) {};
\node at (6,2) [draw,circle]   (S5) {};
\node at (6.5,1) [draw,circle]   (S6) {};
\draw (S1)--(S2);
\draw (S2)--(S3);
\draw (S3)--(S4);
\draw (S4)--(S5);
\end{scope}
\begin{scope}
\node at (6.25,-0.7) {Type I};
\end{scope}

\end{tikzpicture}
\caption{Examples from the four different types of posets of Greene shape $(5,1)$.}
\label{fig:4types}
\end{center}
\end{figure}

Every poset $P$ of Greene shape $(k,1)$ can be encoded as an interval $[a_P, b_P] = I(P)$ of integers from from the set $\{0, 1, \ldots, k+1\}$ as follows.  Number the elements of the spine by $1, \ldots, k$ from bottom to top, and call the foot $e$ as before.  Let $a_P$ (resp.\ $b_P$) be the number of the element of the spine covered by $e$ (resp.\ which covers $e$), or let $a_P = 0$ (resp.\ $b_P=k+1$) if no such element exists.  For example, the four posets $P$ in Figure~\ref{fig:4types} have $I(P)$ equaling $[0,4]$, $[2,6]$, $[2,4]$, and $[0,6]$, respectively.

We are now able to state necessary and sufficient conditions for $F$-positivity, $F$-support containment, $M$-positivity, and $M$-support containment in terms of a simple containment condition for $I(P)$.  Let $\switchx{P}$ denote $P$ with all strict edges, i.e., the result of applying the bar involution to the naturally labeled $P$.

\begin{theorem}\label{thm:k1}
Let $P$ and $Q$ be naturally labeled posets of Greene shape $(k,1)$.  The following are equivalent:
\begin{enumerate}
\renewcommand{\theenumi}{\arabic{enumi}}
\item $P \leqF Q$;\smallskip
\item $\Fsupp(P) \subseteq \Fsupp(Q)$;\smallskip
\item $P \leqM Q$;\smallskip
\item $\switchx{P} \leqM \switchx{Q}$;\smallskip
\item $\Msupp(\switchx{P}) \leqM \Msupp(\switchx{Q})$;\smallskip
\item $I(P) \subseteq I(Q)$.
\end{enumerate}
\end{theorem}

The point of this theorem is that $(6)$ gives an immediate way to determine whether the other five inequalities hold or not.  Note that (5) has $\switchx{P}$ and $\switchx{Q}$ in place of $P$ and $Q$.  This is because, for a naturally labeled poset $P$, the $M$-support of $P$ is the full set of compositions of $\vert P\vert$.

We would like to use jump sequences in our proof, but the jump sequence of a naturally labeled poset $P$ is just $(\vert P\vert )$.  So instead, we will use the jump sequence of $\switchx{P}$, which we call the \emph{bar-jump} of $P$.  Similarly, the \emph{star-bar-jump} of $P$ will mean the jump after applying the star involution to $\switchx{P}$.

\begin{proof}[Proof of Theorem~\ref{thm:k1}]
We know from~\eqref{equ:implications} that (1)$\Rightarrow$(2) and (1)$\Rightarrow$(3).  By Proposition~\ref{pro:fourfold} and using~\eqref{equ:implications} again, we also get that (1)$\Rightarrow$(4)$\Rightarrow$(5).  We will show that (6)$\Rightarrow$(1) and that (2), (3) and (5) each implies (6). 

\textbf{(6)$\Rightarrow$(1).}  For (6) to hold with $P$ not isomorphic to $Q$, we are limited to the following cases.
\begin{enumerate}
\item $Q$ is of Type I.  Then $Q$ is obtained from $P$ by edge deletion, so (1) holds.  
\item $P$ and $Q$ are both of Type D.  Then the foot must be connected to a larger element of the spine in $Q$ than in $P$.  Thus $Q$ is obtainable from $P$ by generalized-redundancy-before-deletion and so, by Theorem~\ref{thm:rbd2}, linear extension containment holds, from which (1) follows.
\item $P$ and $Q$ are both of Type U.  This is similar to the previous case.
\item 
$P$ is of Type B and $Q$ is of Type D or U.  We prove the case when $Q$ is of Type D, with the other case being similar.  Delete the edge in $P$ that goes up from the spine to the foot to obtain an intermediate poset $R$ such that $P \leqF R$.  Since $R$ is of Type D, apply the argument from (b) to get that $R \leqF Q$ and hence (1) holds.  
\item Both $P$ and $Q$ are of Type B.  Then $Q$ is obtained from $P$ by (at most) two redundancy-before-deletion operations, one like in Case (b) and one like in Case (c).  Again, Theorem~\ref{thm:rbd2} yields the desired result.
\end{enumerate}

\textbf{(2)$\Rightarrow$(6).}
Next, suppose (2) holds.  By Corollary~\ref{cor:jump} and the statements that immediately follow it, the bar-jump and star-bar-jump of $Q$ must weakly dominate those of $P$.  
Observe that the bar-jump of a poset $P$ with $I(P) = [a,b]$ is $(1^a, 2, 1^{k-1-a})$ while the star-bar-jump is $(1^{k+1-b}, 2, 1^{b-2})$.  Thus if $I(Q) = [c,d]$, we must have $a \geq c$ and $b\leq d$, so (6) holds.

\textbf{(3)$\Rightarrow$(6).}
If $P$ is a naturally labeled poset with $k+1$ elements, then for any $0 \leq j \leq k+1$, Proposition~\ref{pro:mexpansion} tells us that the coefficient of $M_{(j,k+1-j)}$ in $K_P$ is the number of order ideals with $j$ elements.  If $P$ has Greene shape $(k,1)$ and $I(P) = [a,b]$ with $0 \leq a < b \leq k+1$, we see that the number of such order ideals is 2 if $a < j < b$ and 1 otherwise.  Indeed, $a < j < b$ corresponds to those cases when we have the option of including the foot in the order ideal or not.  Therefore, if $P \leqM Q$, just by considering the coefficients on terms of the form $M_{(j,k+1-j)}$, we obtain $I(P) \subseteq I(Q)$.

\textbf{(5)$\Rightarrow$(6).}
Finally, suppose (5) holds.  By Corollary~\ref{cor:jump} and the sentence immediately before Example~\ref{exa:barjump}, the jump and star-jump of $\switchx{Q}$ must weakly dominate those of $\switchx{P}$.  These are the same conditions as when showing (2)$\Rightarrow$(6), so (6) holds.
\end{proof}

For fixed $k$, considering the set of naturally labeled posets of Greene shape $(k,1)$, we can order them according to $\leqF$ to obtain what we might call the \emph{$F$-positivity poset} for this class.  Note that the latter poset is a poset of posets!  The equivalence of (1) and (6) in Theorem~\ref{thm:k1} shows that the Hasse diagram has a particularly appealing form; see Figure~\ref{fig:fpos} for the case $k=5$, from which the structure for general $k$ is clear.  Of course, we would get the same poset using the orders given by (2)--(5) of Theorem~\ref{thm:k1} in place of $\leqF$.  

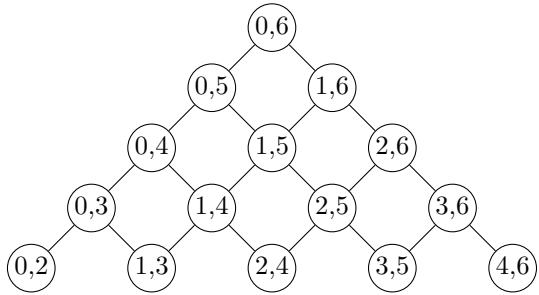
\begin{figure}[htbp]
\begin{center}
\begin{tikzpicture}[scale=0.8]
\tikzstyle{every node} = [draw, shape=circle, inner sep=1pt];
\draw (0,0) node (02) {0,2};
\draw (2,0) node (13) {1,3};
\draw (4,0) node (24) {2,4};
\draw (6,0) node (35) {3,5};
\draw (8,0) node (46) {4,6};
\draw (1,1) node (03) {0,3};
\draw (3,1) node (14) {1,4};
\draw (5,1) node (25) {2,5};
\draw (7,1) node (36) {3,6};
\draw (2,2) node (04) {0,4};
\draw (4,2) node (15) {1,5};
\draw (6,2) node (26) {2,6};
\draw (3,3) node (05) {0,5};
\draw (5,3) node (16) {1,6};
\draw (4,4) node (06) {0,6};
\draw (02) -- (03) -- (13) -- (14) -- (24) -- (25) -- (35) -- (36) -- (46);
\draw (03) -- (04) -- (14) -- (15) -- (25) -- (26) -- (36);
\draw (04) -- (05) -- (15) -- (16) -- (26);
\draw (05) -- (06) -- (16);
\end{tikzpicture}
\end{center}
\caption{The $F$-positivity poset for naturally labeled posets of Greene shape $(5,1)$.  We represent the elements $P$ by $I(P)$ with the square brackets omitted.}
\label{fig:fpos}
\end{figure}

\subsection{Caterpillar posets}

We define a \emph{caterpillar poset} to be a naturally labeled poset consisting of a maximal chain of $k$ elements, again called the \emph{spine}, along with $\ell$ additional elements that are each connected to the spine by at most one edge; an example is shown Figure~\ref{fig:caterpillar}.  Connected caterpillar posets are a poset analogue of caterpillar graphs.   This analogy suggests that we call the non-spine edges ``legs'' and the non-spine elements ``feet.'' 

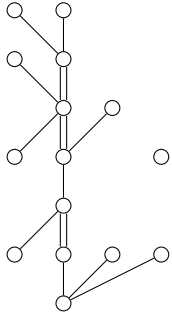
\begin{figure}[ht]
\begin{center}
\begin{tikzpicture}[scale=1.0]
\tikzstyle{every node} = [draw, shape=circle, inner sep=2pt];
\node at (0,0)   (P1) {};
\node at (0,0.5) (P2) {};
\node at (0,1)   (P3) {};
\node at (0,1.5) (P4) {};
\node at (0,2)   (P5) {};
\node at (0,2.5) (P6) {};
\node at (0,3)   (P7) {};

\node at (0.5,0.5) (K1) {};
\node at (-.5,0.5)(K2) {};
\node at (0.5,2)   (K3) {};
\node at (-0.5,3)  (K4) {};
\node at (-0.5,1.5)(K5) {};
\node at (1,1.5) {};
\node at (1,0.5) (K6) {};
\node at (-0.5,2.5) (K7) {};

\draw (K1)--(P1);
\draw (K2)--(P3);
\draw (K3)--(P4);
\draw (K4)--(P6);
\draw (K5)--(P5);
\draw (K7)--(P5);
\draw(K6)--(P1);

\draw (P1)--(P2);
\draw (P2)--(P3);
\draw (P3)--(P4);
\draw (P4)--(P5);
\draw (P5)--(P6);
\draw (P6)--(P7);

\end{tikzpicture}
\caption{An example of a caterpillar poset}
\label{fig:caterpillar}
\end{center}
\end{figure}

The terms ``spine'' and ``feet'' are consistent with the previous subsection, so we can follow the setup there to encode a caterpillar poset as a list of intervals from $\{0,1,\ldots,k+1\}$, with one interval for each foot.  The interval associated with a foot $e$ will be denoted $[a(e), b(e)]$.  Denote the list of these intervals for the feet by $I(P)$ and order it lexicographically.  Thus the poset in Figure~\ref{fig:caterpillar} would be encoded as 
\[
\I(P) = ([0,3],[0,5],[0,8],[1,8],[1,8],[4,8],[5,8],[6,8]).
\]
We will denote the $i$th element of this sequence by $\I(P)_i = [a(p_i), b(p_i)]$, hereafter referring to the foot that contributes $\I(P)_i$ as $p_i$.  Similarly, $q_i$ will contribute $\I(Q)_i$ to $\I(Q)$. 

\begin{theorem}\label{thm:caterpillar}
If $P$ and $Q$ are naturally labeled caterpillar posets with $k$ spine elements and $\ell$ feet then the following are equivalent:
\begin{enumerate}
\renewcommand{\theenumi}{\arabic{enumi}}
\item\label{ite:Fleq} $P \leqF Q$;\smallskip
\item\label{ite:Fsupp} $\Fsupp(P) \subseteq \Fsupp(Q)$;\smallskip
\item\label{ite:Mbarleq} $\switchx{P} \leqM \switchx{Q}$;\smallskip
\item\label{ite:Mbarsupp} $\Msupp(\switchx{P}) \subseteq \Msupp(\switchx{Q})$;\smallskip
\item\label{ite:I}  $\I(P)_i \subseteq \I(Q)_i$ for all $i=1,\ldots,\ell$.
\end{enumerate}
\end{theorem}

Again, the point of this theorem is that the last condition gives a simple way to imply the inequalities shown in the other four conditions.  

\begin{remark} Theorem \ref{thm:caterpillar} does not hold with a mixture of strict and weak edges. For a counterexample, we can consider two posets $P$ and $Q$ that both have $k$ elements on their spine connected by all strict edges, with $\I(P) = \{[1,k+1]\}$ and $\I(Q) = \{[0,k]\}$.  The case $k=3$ is shown in Figure~\ref{fig:counterexample}.  One can check that $K_P=K_Q$ and, in fact, they equal the Schur function $s_{(2,1^{k-1})}$.  The equivalence of \eqref{ite:Fleq} and \eqref{ite:I} would then imply that $\I(P)_1 = \I(Q)_1$ which is false for all $k\geq2$.

\begin{figure}[ht]
\begin{center}
\begin{tikzpicture}[scale=0.8]
\tikzstyle{every node} = [inner sep=2pt]L
\begin{scope}
\node at (0,0) [draw,circle] (P1) {};
\node at (0,1) [draw,circle] (P2) {};
\node at (0,2) [draw,circle] (P3) {};
\node at (0.7,1) [draw,circle] (P4) {};
\draw[double distance=2pt] (P1)--(P2);
\draw[double distance=2pt] (P2)--(P3);
\draw (P1)--(P4);
\node at (-1.6,1) {$P:$};
\end{scope}
\begin{scope}[xshift = 4.2cm]
\node at (0,0) [draw,circle] (P1) {};
\node at (0,1) [draw,circle] (P2) {};
\node at (0,2) [draw,circle] (P3) {};
\node at (-0.7,1) [draw,circle] (P4) {};
\draw[double distance=2pt] (P1)--(P2);
\draw[double distance=2pt] (P2)--(P3);
\draw (P3)--(P4);
\node at (-1.6,1) {$Q:$};
\end{scope}
\end{tikzpicture}
\caption{A counterexample for mixed-edge caterpillar posets.}
\label{fig:counterexample}
\end{center}
\end{figure}
\end{remark}

The feet of a caterpillar poset with $k$ spine elements come in three types, which we will call $D$, $U$, and $I$ to parallel Figure~\ref{fig:4types}.  In other words, Type D are those feet whose legs hang down from the spine, i.e., the foot has interval $[0,i]$ for $2\leq i\leq k$.  Note that $[0,1]$ is not a valid interval since the corresponding foot would be an element of the spine.  Similarly, Type U are those feet whose legs point up from the spine, i.e., the foot has interval $[i,k+1]$ for $1 \leq i \leq k-1$. Type I are isolated feet, meaning they are their own connected component of $P$.  Notice that the lexicographical ordering of $\mathcal{I}(P)$ means that Type D elements are listed first, followed by Type I, and then Type U.  We let $P_D$, $P_U$ and $P_I$ denote the number of elements of a caterpillar poset $P$ of the three types.

Our proofs will make heavy use of the \emph{jump sequence} of a poset as studied in Subsection~\ref{subsec:jump}.  

\begin{lemma}\label{lem:caterpillar}
Suppose $P$ and $Q$ are caterpillar posets with $k$ spine elements and $\ell$ feet.  If $\Fsupp(P) \subseteq \Fsupp(Q)$ or $\Msupp(\switchx{P}) \subseteq \Msupp(\switchx{Q})$ then
\begin{enumerate}
\item $\card{P}{D} + \card{P}{I} \leq \card{Q}{D} + \card{Q}{I}$;\smallskip
\item $\card{P}{U} + \card{P}{I} \leq \card{Q}{U} + \card{Q}{I}$.
\end{enumerate}
\end{lemma}

\begin{proof}
We work with $\switchx{P}$ and $\switchx{Q}$ since  $\Fsupp(P) \subseteq \Fsupp(Q)$ is equivalent to \linebreak $\Fsupp(\switchx{P}) \subseteq \Fsupp(\switchx{Q})$ by Proposition~\ref{pro:fourfold}.  The only elements of $\switchx{P}$ of jump 0 will be those feet of Type D or I along with the bottom element of the spine.  Thus the first element of the jump sequences of $\switchx{P}$ is $\card{P}{D} + \card{P}{I} +1$.  Similarly, the first element of the jump sequence of $\switchx{Q}$ will be $\card{Q}{D} + \card{Q}{I} +1$.  Using Corollary \ref{cor:jump} and inequality (a) that follows it, we yield Lemma~\ref{lem:caterpillar}(a).  By performing the star involution (taking the dual of $P$ and $Q$) and proceeding similarly, we obtain (b).
\end{proof}

\begin{proof}[Proof of Theorem~\ref{thm:caterpillar}]
We know that \eqref{ite:Fleq}$\Rightarrow$\eqref{ite:Fsupp} and \eqref{ite:Mbarleq}$\Rightarrow$\eqref{ite:Mbarsupp}.  We also know from Proposition \ref{pro:fourfold} that \eqref{ite:Fleq} is equivalent to $\switchx{P} \leqF \switchx{Q}$ which then implies \eqref{ite:Mbarleq} by the $M$-positivity of every $F_\alpha$.  We will show that \eqref{ite:I}$\Rightarrow$\eqref{ite:Fleq} and that \eqref{ite:Fsupp} and \eqref{ite:Mbarsupp} each imply \eqref{ite:I}.

\textbf{\eqref{ite:I}$\Rightarrow$\eqref{ite:Fleq}}. 
Assume that $\I(P)_i \subseteq \I(Q)_i$ for all $i=1,\ldots,\ell$.  We wish to transform $P$ into $Q$ by changing the connection of each $p_1, \ldots, p_\ell$ to the spine so that it is the same as that of $q_1, \ldots, q_\ell$, respectively.  Since the feet are not connected directly to each other, these $\ell$ changes can be done in any order and are independent of one another.  So fixing $i$, we wish to show that the result of changing the connection of $p_i$ to the spine to that of $q_i$ results in a poset $P^{(i)}$ such that $P \leqF  P^{(i)}$.

Consider the possible cases for the types of $p_i$ and $q_i$.  If $\I(P)_i = \I(Q)_i$, there is nothing to check.  There are just three cases when $\I(P)_i \subset \I(Q)_i$.
\begin{enumerate}
\item If $q_i$ is of Type I then $P^{(i)}$ is obtained from $P$ by deletion of an edge, so $P \leqF P^{(i)}$.
\item Suppose $p_i$ and $q_i$ are both of Type D.  Hence $q_i$ is connected to a larger element of the spine than $p_i$.  Thus $P^{(i)}$ is obtained from $P$ by generalized-redundancy-before-deletion.  By Theorem~\ref{thm:rbd2}, we get linear extension containment of $P$ in $P^{(i)}$ and so $P \leqF P^{(i)}$. 
\item  The case when $p_i$ and $q_i$ are both of Type U is similar to (b).
\end{enumerate}
The above argument still holds even if the change in the connection of $p_i$ to the spine is applied after changes in the connections of other feet to the spine.  Therefore making the changes from $I(P)_i$ to $I(Q)_i$ for all $i$ results in an increase in $F$-positivity order.

\eqref{ite:Fsupp}$\Rightarrow$\eqref{ite:I} and \eqref{ite:Mbarsupp}$\Rightarrow$\eqref{ite:I}; we will prove these concurrently using the contrapositive.   Suppose $\I(P)_i \not\subseteq \I(Q)_i$ for some $i$.  There are several cases depending on the type of $p_i$ and $q_i$.  Notice that $q_i$ cannot be of Type I since $\I(P)_i \subseteq \I(Q)_i$ always holds in that case.
\begin{enumerate}
\item Both $p_i$ and $q_i$ are of Type U.  We have $\I(P)_i = [a(p_i),k+1]$ and $\I(Q)_i=[a(q_i),k+1]$ with $a(p_i)< a(q_i)$.  We will work with $\switchx{P}$ and $\switchx{Q}$, considering the elements of jump at most $a(p_i)$. The key is that $p_i$ has jump $a(p_i)$ whereas $q_i$ has jump strictly greater than $a(p_i)$.  In full detail, because of the lexicographic order on $\I(P)$, all $i-1$ feet corresponding to the first $i-1$ intervals in $\I(P)$ have jump at most $a(p_i)$ in $\switchx{P}$.  In addition, there are $a(p_i)+1$ spine elements with jump at most $a(p_i)$ and at least one additional foot, including $p_i$, with jump exactly $a(p_i)$.  Thus $\switchx{P}$ has at least $i+a(p_i)+1$ elements of jump at most $a(p_i)$ whereas similar reasoning shows that the number for $\switchx{Q}$ is at most $i + a(p_i)$.  As a result, $\jump(\switchx{P}) \not\domleq \jump(\switchx{Q})$. Using Corollary~\ref{cor:jump} and inequality (a) that follows it, we conclude respectively that $\Msupp(\switchx{P}) \not\subseteq \Msupp(\switchx{Q})$ and $\Fsupp(P) \not\subseteq \Fsupp(Q)$, as required.

\item Both $p_i$ and $q_i$ are of Type D.  We apply the dual operation and follow the steps of the previous part to $\switchx{\rotatex{P}}$ and $\switchx{\rotatex{Q}}$ in place of $\switchx{P}$ and $\switchx{Q}$ to deduce that $\jump(\switchx{\rotatex{P}}) \not\domleq \jump(\switchx{\rotatex{Q}})$.  By the inequality (c) that follows Corollary~\ref{cor:jump}, we conclude that $\Fsupp(P) \not\subseteq \Fsupp(Q)$.  Using the fact that $\switchx{\rotatex{P}}$ is isomorphic to $\rotatep{\switchx{P}}$ and both have all strict edges, we again refer to the paragraph following Corollary~\ref{cor:jump} to conclude that $\Msupp(\switchx{P}) \not\subseteq \Msupp(\switchx{Q})$.

\item $p_i$ is of Type D and $q_i$ is of Type U.  By the lexicographic order on $\I(P)$ and $\I(Q)$, we get $\card{P}{D} \geq i$ but $\card{Q}{D} + \card{Q}{I} < i$, contradicting Lemma~\ref{lem:caterpillar}(a).  Thus $\Fsupp(P) \not\subseteq \Fsupp(Q)$ and $\Msupp(\switchx{P}) \not\subseteq \Msupp(\switchx{Q})$.
\item $p_i$ is of Type U and $q_i$ is of Type D.  Considering $\I(P)$ and $\I(Q)$ from right-to-left, this means $\card{P}{U} \geq \ell+1-i$ but $\card{Q}{U} + \card{Q}{I} < \ell+1-i$, contradicting Lemma~\ref{lem:caterpillar}(b).
\item $p_i$ is of Type I.  If $q_i$ is of Type U, then we get $\card{P}{D} + \card{P}{I} \geq i$ but $\card{Q}{D} + \card{Q}{I} < i$, contradicting Lemma~\ref{lem:caterpillar}(a).   If $q_i$ is of Type D, then we get $\card{P}{U} + \card{P}{I} \geq \ell+1-i$ but $\card{Q}{U} + \card{Q}{I} < \ell+1-i$, contradicting Lemma~\ref{lem:caterpillar}(b).
\end{enumerate}
\end{proof}

\begin{remark}
Notably absent from Theorem~\ref{thm:caterpillar} is the condition $P \leqM Q$.  It is absent because it is not equivalent; the smallest example of a pair $P$ and $Q$ such that $P \leqM Q$ but $P \not\leqF Q$ is shown in Figure~\ref{fig:Mcounterexample}.  We can compute that $K_Q - K_P = F_{41} + F_{32} + F_{311} + F_{221} - F_{122}$ and is $M$-positive.  If we would prefer both $P$ and $Q$ to be connected, there exist such counterexamples using 6-element caterpillars.   

\begin{figure}[ht]
\begin{center}
\begin{tikzpicture}[scale=0.8]
\tikzstyle{every node} = [inner sep=2pt]L
\begin{scope}
\node at (0,0) [draw,circle] (P1) {};
\node at (0,1) [draw,circle] (P2) {};
\node at (0,2) [draw,circle] (P3) {};
\node at (0.7,1) [draw,circle] (P4) {};
\node at (-0.7,0) [draw,circle] (P5) {};
\draw (P1)--(P2);
\draw (P2)--(P3);
\draw (P1)--(P4);
\draw (P2)--(P5);
\node at (-1.6,1) {$P:$};
\end{scope}
\begin{scope}[xshift = 4.2cm]
\node at (0,0) [draw,circle] (P1) {};
\node at (0,1) [draw,circle] (P2) {};
\node at (0,2) [draw,circle] (P3) {};
\node at (0.7,2) [draw,circle] (P4) {};
\node at (1,1) [draw,circle] (P5) {};
\draw(P1)--(P2);
\draw (P2)--(P3);
\draw (P2)--(P4);
\node at (-1.6,1) {$Q:$};
\end{scope}
\end{tikzpicture}
\caption{An example where $P \leqM Q$ but $P \not\leqF Q$}
\label{fig:Mcounterexample}
\end{center}
\end{figure}
\end{remark}

Recent work, such as \cite{AAM24,ADM23,HaTs17,LiWe20a,LiWe20b,McWa14,Zho20+}, has looked for classes of posets where the elements are \emph{distinguished} by $K_P$, meaning that if $P$ and $Q$ are not isomorphic, then $K_P \neq K_Q$.  

\begin{corollary}
$K_P$ distinguishes the class of caterpillar posets.
\end{corollary}

\begin{proof}
Suppose $K_P = K_Q$ for caterpillars $P$ and $Q$.  Clearly $P$ and $Q$ must have the same number of elements.   By \cite[Prop.~3.7]{McWa14}, we have $K_{\switchx{P}}= K_{\switchx{Q}}$\,.  Then \cite[Prop.~4.2]{McWa14} tells us that the jump sequences of $\switchx{P}$ and $\switchx{Q}$ have the same length, which implies that the lengths of the spines of $\switchx{P}$ and $\switchx{Q}$ are equal.  Thus $P$ and $Q$ also have the same number of feet.  Theorem~\ref{thm:caterpillar} then implies that $\I(P) = \I(Q)$ and so $P$ and $Q$ are isomorphic.
\end{proof}  

\section{Open problems}\label{sec:conclusion}

We have not attempted to be comprehensive and there are many further avenues of investigation for comparing $\kpw$ and $\kqt$, whether according to positivity or support containment.  In addition to Question~\ref{que:longestchain}, let us mention a selection of other questions that arose during our investigations, divided into 4 types.

\begin{enumerate}
\renewcommand{\theenumi}{\arabic{enumi}}

\item A weakness of our necessary conditions in Section~\ref{sec:nec} is that most of those for generally labeled posets only require us to assume $M$-support containment.  Are there stronger necessary conditions that are based on $M$-positivity, $F$-support containment, or especially $F$-positivity?  

\item A useful source of ideas for us was to construct the entire $F$-positivity poset for small values of $\vert P\vert $ and see if our results could explain all the cover relations, and we recommend this approach for others interested in working in this area. 

 Let us mention a technique that explains some of the cover relations we encountered that are not explained by the results so far.  If a labeled poset $(P,\omega)$ has two incomparable elements $x$ and $y$, its set of $(P,\omega)$-partitions $\popf$ can be partitioned into two subsets: those where $\popf(x) \geq \popf(y)$ and those where $\popf(x) < \popf(y)$.  We use this idea to explain inequalities such as 
\begin{equation}\label{equ:decomp1}
\begin{tikzpicture}[scale=0.7]
\tikzstyle{every node} = [inner sep=2pt]
\node at (0,0.5) [draw,circle] (P5) {};
\node at (0.7,-0.5) [draw,circle] (P6) {};
\node at (0.7,0.5) [draw,circle] (P7) {};
\node at (0.7,1.5) [draw,circle] (P8) {};
\draw (P6)--(P7);
\draw (P7)--(P8);

\node at (2,0.5) {$\leq_F$};

\node at (3.2,0) [draw,circle] (P1) {};
\node at (3.2,1) [draw,circle] (P2) {};
\node at (4.2,0) [draw,circle] (P3) {};
\node at (4.2,1) [draw,circle] (P4) {};
\draw (P1)--(P2);
\draw (P2)--(P3);
\draw (P3)--(P4);

\node at (4.7,0.5) {.};
\end{tikzpicture} 
\end{equation}
Indeed, the lesser poset can be decomposed as 
\begin{equation}\label{equ:decomp2}
\begin{tikzpicture}[scale=0.7]
\tikzstyle{every node} = [inner sep=2pt]
\node at (1.3,1) [draw,circle] (P1) {};
\node at (2,0) [draw,circle] (P2) {};
\node at (2,1) [draw,circle] (P3) {};
\node at (2,2) [draw,circle] (P4) {};
\node at (0.9,1) {$y$};
\node at (2.4,2) {$x$};
\draw (P2)--(P3);
\draw (P3)--(P4);

\node at (3,1) {$=$};

\node at (4,1) [draw,circle] (P5) {};
\node at (4.5,2) [draw,circle] (P6) {};
\node at (5,1) [draw,circle] (P7) {};
\node at (5,0) [draw,circle] (P8) {};
\node at (3.6,1) {$y$};
\node at (4.9,2) {$x$};
\draw (P5)--(P6);
\draw (P6)--(P7);
\draw (P7)--(P8);

\node at (6,1) {$+$};

\node at (7,-0.5) [draw,circle] (P9) {};
\node at (7,0.5) [draw,circle] (P10) {};
\node at (7,1.5) [draw,circle] (P11) {};
\node at (7,2.5) [draw,circle] (P12) {};
\node at (6.6,2.5) {$y$};
\node at (7.4,1.5) {$x$};
\draw (P9)--(P10);
\draw (P10)--(P11);
\draw[double distance=2pt] (P11)--(P12);
\end{tikzpicture}
\end{equation}
(where each labeled poset $(P,\om)$ is represting $\kpw$ in this algebraic expression)
and the greater poset can be decomposed as
\begin{equation}\label{equ:decomp3}
\begin{tikzpicture}[scale=0.7]
\tikzstyle{every node} = [inner sep=2pt]
\node at (1,0.5) [draw,circle] (P1) {};
\node at (1,1.5) [draw,circle] (P2) {};
\node at (2,0.5) [draw,circle] (P3) {};
\node at (0.6,1.5) {$x$};
\node at (2.4,1.5) {$y$};
\node at (2,1.5) [draw,circle] (P4) {};
\draw (P1)--(P2);
\draw (P2)--(P3);
\draw (P3)--(P4);

\node at (3,1) {$=$};

\node at (4,1) [draw,circle] (P5) {};
\node at (4.5,2) [draw,circle] (P6) {};
\node at (4.1,2) {$x$};
\node at (5,1) [draw,circle] (P7) {};
\node at (5.4,1) {$y$};
\node at (5,0) [draw,circle] (P8) {};
\draw (P5)--(P6);
\draw (P6)--(P7);
\draw (P7)--(P8);

\node at (6.2,1) {$+$};

\node at (7,0) [draw,circle] (P9) {};
\node at (8,0) [draw,circle] (P10) {};
\node at (7.5,1) [draw,circle] (P11) {};
\node at (7.1,1) {$x$};
\node at (7.9,2) {$y$};
\node at (7.5,2) [draw,circle] (P12) {};
\draw (P9)--(P11);
\draw (P10)--(P11);
\draw[double distance=2pt] (P11)--(P12);
\node at (8.5,1) {.};
\end{tikzpicture} 
\end{equation}
These two decompositions explain the inequality in~\eqref{equ:decomp1} since the rightmost labeled poset in~\eqref{equ:decomp3} can be obtained from the rightmost poset in~\eqref{equ:decomp2} by generalized-redundancy-before-deletion.

However, it is important not to overuse this technique in the following sense.  By repeated applications of such decompositions to labeled posets $(P,\om)$ and $(Q,\tau)$, one can eventually arrive at two sums of labeled chains which we can try to compare.  But these sums of chains are just the graphical representations of the $F$-expansions of $\kpw$ and $\kqt$, so this approach is doing nothing more than comparing $\kpw$ and $\kqt$ by explicitly computing their $F$-expansions.

\item There are ways we could attempt to strengthen the results of Section~\ref{sec:greene}.
\begin{itemize}
\item Perhaps the most obvious and compelling is to find larger classes of posets for which there are simple conditions that are both necessary and sufficient for $F$-positivity.  A class that would generalize and unify the posets of Greene shape $(k,1)$ of Theorem~\ref{thm:k1} and the caterpillar posets of Theorem~\ref{thm:caterpillar} would be ``caterpillar-like" posets that also allow feet of Type B (as defined in Figure~\ref{fig:4types}). 
\item Theorem~\ref{thm:caterpillar} suggests the question of what happens when we allow both strict and weak edges in the caterpillar graph.  Relations such as 
\begin{equation}\label{eqn:nonident_spine_counter}
\begin{tikzpicture}[scale=0.8]
\tikzstyle{every node} = [inner sep=2pt]
\begin{scope}[xshift = 3.2cm]
\node at (0,0) [draw,circle] (P1) {};
\node at (0,1) [draw,circle] (P2) {};
\node at (0,2) [draw,circle] (P3) {};
\node at (-0.7,1) [draw,circle] (P4) {};
\node at (0.7,1) [draw,circle] (P5) {};
\draw (P1)--(P2);
\draw (P2)--(P3);
\draw (P4)--(P3);
\draw (P5)--(P3);
\end{scope}
\begin{scope}
\node at (0,0) [draw,circle] (P1) {};
\node at (0,1) [draw,circle] (P2) {};
\node at (0,2) [draw,circle] (P3) {};
\node at (-0.7,0) [draw,circle] (P4) {};
\node at (0.7,0) [draw,circle] (P5) {};
\draw[double distance=2pt] (P1)--(P2);
\draw (P2)--(P3);
\draw (P4)--(P2);
\draw (P5)--(P2);
\node at (1.6,1) {$\leq_F$};
\end{scope}
\end{tikzpicture}
\end{equation}
mean that we can still get comparability.  Can we properly explain such inequalities?

\end{itemize}

\item 

We have focused largely on the case when the assignment of strict/weak edges is arbitrary.  But in the equality question for $\kpw$, it is possible to get stronger results by restricting one's attention to naturally labeled posets, as is done in \cite{HaTs17,LiWe20a,LiWe20b,McWa14} and as we did here in Subsection~\ref{sub:convex} and Section~\ref{sec:greene}.  Are there other results in this paper that can be advanced or new types of results that can be obtained by restricting to the naturally labeled case?  The necessary conditions of \cite[Section~3]{LiWe20a} might be a good source of ideas.\end{enumerate}

\bibliography{ppartitionpositivity}

\newcommand{\etalchar}[1]{$^{#1}$}
\newcommand{\noopsort}[1]{} \newcommand{\printfirst}[2]{#1}
  \newcommand{\singleletter}[1]{#1} \newcommand{\switchargs}[2]{#2#1}
\begin{thebibliography}{HLMvW11}

\bibitem[AAM24]{AAM24}
Doriann Albertin, Jean-Christophe Aval, and Hugo Mlodecki.
\newblock Quasisymmetric invariants for families of posets.
\newblock {\em Australas. J. Combin.}, 90:341--356, 2024.

\bibitem[AB20]{AsBe20}
Sami Assaf and Nantel Bergeron.
\newblock Flagged ( {P}, {$\rho$} ) -partitions.
\newblock {\em European J. Combin.}, 86:103085, 2020.

\bibitem[ADM23]{ADM23}
Jean-Christophe Aval, Karimatou Djenabou, and Peter R.~W. McNamara.
\newblock Quasisymmetric functions distinguishing trees.
\newblock {\em Algebr. Comb.}, 6(3):595--614, 2023.

\bibitem[BBR06]{BBR06}
Fran{\c{c}}ois Bergeron, Riccardo Biagioli, and Mercedes~H. Rosas.
\newblock Inequalities between {L}ittlewood-{R}ichardson coefficients.
\newblock {\em J. Combin. Theory Ser. A}, 113(4):567--590, 2006.

\bibitem[BBS{\etalchar{+}}14]{BBSSZ14}
Chris Berg, Nantel Bergeron, Franco Saliola, Luis Serrano, and Mike Zabrocki.
\newblock A lift of the {S}chur and {H}all-{L}ittlewood bases to
  non-commutative symmetric functions.
\newblock {\em Canad. J. Math.}, 66(3):525--565, 2014.

\bibitem[BHK18a]{BHK18+}
Thomas Browning, Max Hopkins, and Zander Kelley.
\newblock Doppelgangers: the {U}r-operation and posets of bounded height.
\newblock \url{https://arxiv.org/abs/1710.10407}, 2018+.

\bibitem[BHK18b]{BHK18}
Thomas Browning, Max Hopkins, and Zander Kelley.
\newblock Doppelgangers: the {U}r-operation and posets of bounded height
  (extended abstract).
\newblock {\em S\'{e}m. Lothar. Combin.}, 80B:Art. 80, 12, 2018.

\bibitem[BO14]{BaOr14}
Cristina Ballantine and Rosa Orellana.
\newblock Schur-positivity in a square.
\newblock {\em Electron. J. Combin.}, 21(3):Paper 3.46, 36, 2014.

\bibitem[CM18]{CaMe18}
Erik Carlsson and Anton Mellit.
\newblock A proof of the shuffle conjecture.
\newblock {\em J. Amer. Math. Soc.}, 31(3):661--697, 2018.

\bibitem[DHT02]{DHT02}
G{\'e}rard Duchamp, Florent Hivert, and Jean-Yves Thibon.
\newblock Noncommutative symmetric functions. {VI}. {F}ree quasi-symmetric
  functions and related algebras.
\newblock {\em Internat. J. Algebra Comput.}, 12(5):671--717, 2002.

\bibitem[DKLT96]{DKLT96}
G{\'e}rard Duchamp, Daniel Krob, Bernard Leclerc, and Jean-Yves Thibon.
\newblock Fonctions quasi-sym\'etriques, fonctions sym\'etriques non
  commutatives et alg\`ebres de {H}ecke \`a {$q=0$}.
\newblock {\em C. R. Acad. Sci. Paris S\'er. I Math.}, 322(2):107--112, 1996.

\bibitem[DP07]{DoPy07}
Galyna Dobrovolska and Pavlo Pylyavskyy.
\newblock On products of {$\mathfrak{sl}\sb \mathfrak{n}$} characters and
  support containment.
\newblock {\em J. Algebra}, 316(2):706--714, 2007.

\bibitem[Ehr96]{Ehr96}
Richard Ehrenborg.
\newblock On posets and {H}opf algebras.
\newblock {\em Adv. Math.}, 119(1):1--25, 1996.

\bibitem[F{\'{e}}r15]{Fer15}
Valentin F{\'{e}}ray.
\newblock Cyclic inclusion-exclusion.
\newblock {\em SIAM J. Discrete Math.}, 29(4):2284--2311, 2015.

\bibitem[FFLP05]{FFLP05}
Sergey Fomin, William Fulton, Chi-Kwong Li, and Yiu-Tung Poon.
\newblock Eigenvalues, singular values, and {L}ittlewood-{R}ichardson
  coefficients.
\newblock {\em Amer. J. Math.}, 127(1):101--127, 2005.

\bibitem[Ges84]{Ges84}
Ira~M. Gessel.
\newblock Multipartite {$P$}-partitions and inner products of skew {S}chur
  functions.
\newblock In {\em Combinatorics and algebra (Boulder, Colo., 1983)}, volume~34
  of {\em Contemp. Math.}, pages 289--301. Amer. Math. Soc., Providence, RI,
  1984.

\bibitem[Ges90]{Ges90}
Ira~M. Gessel.
\newblock Quasi-symmetric functions.
\newblock Unpublished manuscript, 1990.

\bibitem[Gre76]{Gre76}
Curtis Greene.
\newblock Some partitions associated with a partially ordered set.
\newblock {\em J. Combinatorial Theory Ser. A}, 20(1):69--79, 1976.

\bibitem[Hag04]{Hag04}
J.~Haglund.
\newblock A combinatorial model for the {M}acdonald polynomials.
\newblock {\em Proc. Natl. Acad. Sci. USA}, 101(46):16127--16131, 2004.

\bibitem[HHL05a]{HHL05}
J.~Haglund, M.~Haiman, and N.~Loehr.
\newblock A combinatorial formula for {M}acdonald polynomials.
\newblock {\em J. Amer. Math. Soc.}, 18(3):735--761, 2005.

\bibitem[HHL{\etalchar{+}}05b]{HHLRU05}
J.~Haglund, M.~Haiman, N.~Loehr, J.~B. Remmel, and A.~Ulyanov.
\newblock A combinatorial formula for the character of the diagonal
  coinvariants.
\newblock {\em Duke Math. J.}, 126(2):195--232, 2005.

\bibitem[HLMvW11]{HLMvW11}
J.~Haglund, K.~Luoto, S.~Mason, and S.~van Willigenburg.
\newblock Quasisymmetric {S}chur functions.
\newblock {\em J. Combin. Theory Ser. A}, 118(2):463--490, 2011.

\bibitem[HT17]{HaTs17}
Takahiro Hasebe and Shuhei Tsujie.
\newblock Order quasisymmetric functions distinguish rooted trees.
\newblock {\em J. Algebraic Combin.}, 46(3-4):499--515, 2017.

\bibitem[Kir04]{Kir04}
Anatol~N. Kirillov.
\newblock An invitation to the generalized saturation conjecture.
\newblock {\em Publ. Res. Inst. Math. Sci.}, 40(4):1147--1239, 2004.

\bibitem[KT97]{KrTh97}
Daniel Krob and Jean-Yves Thibon.
\newblock Noncommutative symmetric functions. {IV}. {Q}uantum linear groups and
  {H}ecke algebras at {$q=0$}.
\newblock {\em J. Algebraic Combin.}, 6(4):339--376, 1997.

\bibitem[KWvW08]{KWvW08}
Ronald~C. King, Trevor~A. Welsh, and Stephanie~J. van Willigenburg.
\newblock Schur positivity of skew {S}chur function differences and
  applications to ribbons and {S}chubert classes.
\newblock {\em J. Algebraic Combin.}, 28(1):139--167, 2008.

\bibitem[LLT97]{LLT97}
Alain Lascoux, Bernard Leclerc, and Jean-Yves Thibon.
\newblock Ribbon tableaux, {H}all-{L}ittlewood functions, quantum affine
  algebras, and unipotent varieties.
\newblock {\em J. Math. Phys.}, 38(2):1041--1068, 1997.

\bibitem[LM22]{LeMc22}
Nathan R.~T. Lesnevich and Peter R.~W. McNamara.
\newblock Positivity among {$P$}-partition generating functions.
\newblock {\em Ann. Comb.}, 26(1):171--204, 2022.

\bibitem[LM25]{LeMc22_correction}
Nathan R.~T. Lesnevich and Peter R.~W. McNamara.
\newblock Correction to: Positivity among {$P$}-partition generating functions.
\newblock {\em Ann. Comb.}, 2025.
\newblock https://doi.org/10.1007/s00026-025-00744-3.

\bibitem[LMvW13]{LMvW13}
Kurt Luoto, Stefan Mykytiuk, and Stephanie van Willigenburg.
\newblock {\em An introduction to quasisymmetric {S}chur functions}.
\newblock SpringerBriefs in Mathematics. Springer, New York, 2013.

\bibitem[LP08]{LaPy08}
Thomas Lam and Pavlo Pylyavskyy.
\newblock {$P$}-partition products and fundamental quasi-symmetric function
  positivity.
\newblock {\em Adv. in Appl. Math.}, 40(3):271--294, 2008.

\bibitem[LPP07]{LPP07}
Thomas Lam, Alexander Postnikov, and Pavlo Pylyavskyy.
\newblock Schur positivity and {S}chur log-concavity.
\newblock {\em Amer. J. Math.}, 129(6):1611--1622, 2007.

\bibitem[LW20a]{LiWe20a}
Ricky~Ini Liu and Michael Weselcouch.
\newblock {$P$}-partition generating function equivalence of naturally labeled
  posets.
\newblock {\em J. Combin. Theory Ser. A}, 170:105136, 31, 2020.

\bibitem[LW20b]{LiWe20b}
Ricky~Ini Liu and Michael Weselcouch.
\newblock ${P}$-partitions and quasisymmetric power sums.
\newblock {\em Int. Math. Res. Not. IMRN}, 02 2020.
\newblock rnz375.

\bibitem[Mal93]{MalThesis}
Claudia Malvenuto.
\newblock {\em Produits et coproduits des fonctions quasi-sym\'etriques et de
  l'alg\`ebre des descentes}, volume~16 of {\em Publications du Laboratoire de
  Combinatoire et d'Informatique Math\'ematique}.
\newblock Laboratoire de Combinatoire et d'Informatique Math\'ematique,
  Universit\'e du Qu\'ebec \`a Montr\'eal, 1993.
\newblock Ph.D. thesis.

\bibitem[McN14]{McN14}
Peter R.~W. McNamara.
\newblock Comparing skew {S}chur functions: a quasisymmetric perspective.
\newblock {\em J. Comb.}, 5(1):51--85, 2014.

\bibitem[MR95]{MaRe95}
Claudia Malvenuto and Christophe Reutenauer.
\newblock Duality between quasi-sym\-met\-ric functions and the {S}olomon
  descent algebra.
\newblock {\em J. Algebra}, 177(3):967--982, 1995.

\bibitem[MR98]{MaRe98}
Claudia Malvenuto and Christophe Reutenauer.
\newblock Plethysm and conjugation of quasi-symmetric functions.
\newblock {\em Discrete Math.}, 193(1-3):225--233, 1998.
\newblock Selected papers in honor of Adriano Garsia (Taormina, 1994).

\bibitem[MvW09]{McvW09b}
Peter R.~W. McNamara and Stephanie van Willigenburg.
\newblock Positivity results on ribbon {S}chur function differences.
\newblock {\em European J. Combin.}, 30(5):1352--1369, 2009.

\bibitem[MvW12]{McvW12}
Peter R.~W. McNamara and Stephanie van Willigenburg.
\newblock Maximal supports and {S}chur-positivity among connected skew shapes.
\newblock {\em European J. Combin.}, 33(6):1190--1206, 2012.

\bibitem[MW14]{McWa14}
Peter R.~W. McNamara and Ryan~E. Ward.
\newblock Equality of {$P$}-partition generating functions.
\newblock {\em Ann. Comb.}, 18(3):489--514, 2014.

\bibitem[Nor79]{Nor79}
P.~N. Norton.
\newblock {$0$}-{H}ecke algebras.
\newblock {\em J. Austral. Math. Soc. Ser. A}, 27(3):337--357, 1979.

\bibitem[Oko97]{Oko97}
Andrei Okounkov.
\newblock Log-concavity of multiplicities with application to characters of
  {${\rm U}(\infty)$}.
\newblock {\em Adv. Math.}, 127(2):258--282, 1997.

\bibitem[Pec20]{Pec20}
O.~Pechenik.
\newblock The {G}enomic {S}chur {F}unction is {F}undamental-{P}ositive.
\newblock {\em Ann. Comb.}, 24(1):95--108, 2020.

\bibitem[RS06]{RhSk06}
Brendon Rhoades and Mark Skandera.
\newblock Kazhdan-{L}usztig immanants and products of matrix minors.
\newblock {\em J. Algebra}, 304(2):793--811, 2006.

\bibitem[RS10]{RhSk10}
Brendon Rhoades and Mark Skandera.
\newblock Kazhdan-{L}usztig immanants and products of matrix minors. {II}.
\newblock {\em Linear Multilinear Algebra}, 58(1-2):137--150, 2010.

\bibitem[Sta71]{StaThesis71}
Richard~P. Stanley.
\newblock {\em Ordered structures and partitions}.
\newblock PhD thesis, Harvard University, 1971.

\bibitem[Sta72]{StaThesis}
Richard~P. Stanley.
\newblock {\em Ordered structures and partitions}.
\newblock American Mathematical Society, Providence, R.I., 1972.
\newblock Memoirs of the American Mathematical Society, No. 119.

\bibitem[Sta99]{ec2}
Richard~P. Stanley.
\newblock {\em Enumerative combinatorics. {V}ol. 2}, volume~62 of {\em
  Cambridge Studies in Advanced Mathematics}.
\newblock Cambridge University Press, Cambridge, 1999.

\bibitem[Sta12]{ec1e2}
Richard~P. Stanley.
\newblock {\em Enumerative combinatorics. {V}olume 1}, volume~49 of {\em
  Cambridge Studies in Advanced Mathematics}.
\newblock Cambridge University Press, Cambridge, second edition, 2012.

\bibitem[{The}19]{SageMath}
{The Sage Developers}.
\newblock {\em {S}ageMath, the {S}age {M}athematics {S}oftware {S}ystem
  ({V}ersion 8.8)}, 2019.
\newblock {\tt https://www.sagemath.org}.

\bibitem[TvW18]{TovW18}
Foster Tom and Stephanie van Willigenburg.
\newblock Necessary conditions for {S}chur maximality.
\newblock {\em Electron. J. Combin.}, 25(2):Paper 2.30, 50 pp., 2018.

\bibitem[Zho20]{Zho20+}
Jeremy Zhou.
\newblock Reconstructing rooted trees from their strict order quasisymmetric
  functions.
\newblock \href{arXiv:2008.00424}{https://arxiv.org/abs/2008.00424}, 2020+.

\end{thebibliography}
\bibliographystyle{alpha}

\end{document}